\UseRawInputEncoding
\documentclass{amsart}
\usepackage[utf8]{inputenc}
\usepackage[english]{babel}
\usepackage{url}
\usepackage{hyperref}
\usepackage{graphicx}
\usepackage{amsmath}%
\usepackage{amsfonts}%
\usepackage{latexsym}
\usepackage{amssymb, amsmath}
\usepackage{epsfig}%
\usepackage{mathrsfs}%
\usepackage{longtable}%
\usepackage{setspace}
\usepackage{geometry, calc, color}
\usepackage{indentfirst}%
\usepackage{wrapfig}%
\usepackage{amsthm}
\usepackage{enumitem} 
\usepackage{subfigure} 
\usepackage{subfiles}
\usepackage{float}

\newcommand{\C}{\mathbb{C}}

\newcommand{\R}{\mathbb{R}}
\newcommand{\Z}{\mathbb{Z}}
\newcommand{\Q}{\mathbb{Q}}
\newcommand{\N}{\mathbb{N}}
\newcommand{\M}{\mathcal{M}}

\newcommand{\J}{\mathcal{J}}

\renewcommand{\P}{\mathcal{P}}

\newcommand{\cz}{{\rm CZ}}

\newcommand{\Ker}{\mathrm{ker}}
\newcommand{\eo}{\mathrm{eo}}

\newtheorem{theorem}{Theorem}[section]

\newtheorem{lemma}[theorem]{Lemma}
\newtheorem{prop}[theorem]{Proposition}

\newtheorem{definition}[theorem]{Definition}

\newtheorem{rem}[theorem]{Remark}

\newcommand{\CZ}{\mu_{CZ}^{\tau}}
\newcommand{\wind}{\text{wind}^{\tau}}
\newcommand{\windd}{\text{wind}}
\newcommand{\ind}{\textnormal{ind}}
\newcommand{\coker}{\textnormal{coker}}

\begin{document}

\title{Cylindrical contact homology for weakly convex  contact forms in dimension three}

\author[Ana Kelly de Oliveira]{Ana Kelly de Oliveira} 
\address{Ana Kelly de Oliveira, Shenzhen International Center for Mathematics, Southern University of Science and Technology, Shenzhen, China}
\email{ako\_kelly@hotmail.com}

\author[Pedro A. S. Salom\~ao]{Pedro A. S. Salom\~ao}
\address{Pedro A. S. Salom\~ao, Shenzhen International Center for Mathematics, Southern University of Science and Technology, Shenzhen, China}
\email{psalomao@sustech.edu.cn}

\date{}

\begin{abstract}
A contact form $\lambda$ on a closed contact three-manifold $(M,\xi)$ is called weakly convex if either it has no contractible Reeb orbit, or the first Chern class of $\xi$ vanishes on $\pi_2(M)$  and the index of every contractible Reeb orbit is at least $2$. We present conditions for a weakly convex contact form to admit a well-defined cylindrical contact homology. The key point is a cancellation mechanism for boundary degenerations involving 
index-2 Reeb orbits, based on a parity property of holomorphic planes. 
\end{abstract}

\maketitle

\tableofcontents

\section{Introduction and main results}

Cylindrical contact homology is an invariant of contact manifolds
defined in terms of the dynamics of their Reeb flows. Its chain complex
$(C,\partial)$ is generated by good Reeb orbits, and the differential
$\partial$ counts rigid pseudo-holomorphic
cylinders connecting such generators.
The definition of the differential is subtle. In general, transversality
may fail especially in the presence of multiply covered curves, and compactness issues may arise
from the presence of additional holomorphic buildings in the boundary of
the moduli spaces. These phenomena can obstruct the proof that
$\partial$ is well-defined and satisfies $\partial^2=0$. 

In \cite{HN}, Hutchings and Nelson showed that cylindrical contact homology is generically well-defined under the dynamically convex
assumption, which excludes certain low-index contractible Reeb orbits
and prevents the appearance of problematic holomorphic buildings.

\begin{theorem}[Hutchings-Nelson \cite{HN}]\label{thm:HN} Let $\lambda$ be a  nondegenerate contact form on a closed contact three-manifold $($M$,\xi)$. Assume that the Reeb flow of $\lambda$ either has no contractible Reeb orbit, or the first Chern class $c_1(\xi)$ vanishes on $\pi_2(M)$ and the index of every contractible Reeb orbit is at least $3$. Assume further in the latter case that every index-$3$ contractible Reeb orbit is embedded.  
Then, the cylindrical contact homology is well-defined for generic $\lambda$-compatible almost complex structures $J$ on $\R \times M$. 
\end{theorem}


The main goal of this paper is to extend Theorem \ref{thm:HN} to the weakly convex
case, where contractible Reeb orbits of Conley–Zehnder index $2$ are allowed.
In this case, the compactification of the moduli spaces of index-$2$
holomorphic cylinders may contain additional holomorphic buildings arising
from bubbling at index-$2$ Reeb orbits. These configurations are not present
in the dynamically convex setting and constitute the main obstruction to
the definition of the differential. We show that these obstructions can be completely controlled by a parity phenomenon for holomorphic planes asymptotic to index-2 orbits, leading to a cancellation mechanism for the boundary contributions, as such planes appear in pairs with opposite asymptotic directions.

We first formulate an abstract criterion ensuring that cylindrical contact homology is well-defined.
This criterion is expressed in terms of a parity condition on the moduli
spaces of rigid holomorphic planes asymptotic to index-$2$ orbits.

\begin{theorem}\label{main1}
Let $\lambda$ be a nondegenerate weakly convex contact form on a closed
contact three-manifold $(M,\xi)$ and assume that $c_1(\xi)|_{\pi_2(M)}=0$.
Assume that every contractible index-$2$ and index-$3$ Reeb orbit is embedded. Assume furthermore that the following parity condition holds for a generic $\lambda$-compatible almost complex structure $J\in \mathcal J(\lambda)$ on $\R\times M$: For every simple contractible Reeb orbit 
$\alpha$ with $\mu_{CZ}(\alpha)=2$, the moduli space $
\mathcal{M}^J_{0,1}(\alpha;\emptyset)/\mathbb{R}
$
of $J$--holomorphic planes asymptotic to $\alpha$ is finite and has even
cardinality. Moreover, these planes can be partitioned into two subsets
of equal cardinality, according to their asymptotic direction at $\alpha$,
so that half of them approach $\alpha$ through one direction and half
through the opposite direction.
 Then, for such generic $J$, the cylindrical contact homology of $(M,\lambda)$
is well-defined.
\end{theorem}

The parity condition in Theorem \ref{main1} is analytic in nature and a priori
depends on the structure of the moduli spaces of holomorphic planes.
In order to obtain a more geometric and topological criterion,
we present a particular formulation of the main result in which
the parity property is derived from global assumptions on the
contact manifold and its filling.
More precisely, under suitable topological conditions on $(M,\xi)$
and geometric conditions on a strong symplectic filling of $(M,\xi)$,
we prove that every simple contractible index-$2$ Reeb orbit admits
either no rigid holomorphic plane or precisely two such planes,
approaching the orbit through opposite asymptotic directions.
In particular, the parity hypothesis in Theorem \ref{main1} is automatically satisfied.

\begin{theorem}\label{main2}
Let $\lambda$ be a nondegenerate weakly convex contact form on a closed
contact three--manifold $(M,\xi)$ satisfying $\pi_2(M)=0$. 
Assume that $(M,\xi)$ admits a strong symplectic filling $(W,\omega)$
satisfying $\omega|_{\pi_2(W)}=0,$
and the inclusion
$\pi_1(M)\to\pi_1(W)$
is injective.
Assume furthermore that
\begin{itemize}
\item every contractible index-$2$ Reeb orbit is simple and unknotted,
\item every contractible index-$3$ Reeb orbit is embedded.
\end{itemize}
Then for generic $\lambda$-compatible almost complex structures
$J\in\mathcal{J}_{\mathrm{reg}}(\lambda)$ the following assertions hold:

\begin{itemize}
\item[(i)] For every simple contractible Reeb orbit $\alpha$ with
$\mu_{CZ}(\alpha)=2$, the moduli space
$
\mathcal{M}^J_{0,1}(\alpha;\emptyset)/\mathbb{R}
$
is either empty or consists of precisely two elements,
whose asymptotic approaches to $\alpha$ are opposite.

\item[(ii)] The cylindrical contact homology of $(M,\lambda)$ is well-defined.
\end{itemize}
\end{theorem}

The main difficulty in the proof of Theorem~\ref{main1} is the
appearance of non-regular holomorphic curves in the compactification of
the moduli spaces of index-$2$ cylinders, together with the bubbling-off
phenomenon at contractible index-$2$ Reeb orbits. In particular, one
may find holomorphic buildings with two levels: the upper level
contains a multiply covered curve of negative Fredholm index whose
underlying simple curve is a pair of pants, while the lower level
consists of finitely many rigid holomorphic planes asymptotic to the same
index-$2$ orbit, together with a trivial cylinder. To treat this configuration, we adapt to our setting the gluing theory
for non-regular holomorphic curves developed in \cite{BH18} and \cite{HT07}. The multiply covered curve in the upper level can be
glued with the rigid planes in the lower level that approach the orbit
through the opposite asymptotic direction. The resulting glued curves
provide a natural continuation of the moduli space of index-$2$
cylinders across the boundary corresponding to such buildings.

Since we work with coefficients in $\mathbb{Z}$, it is crucial to
understand the behavior of orientations under this gluing procedure.
A careful analysis shows that when the upper-level curve is glued to
opposite rigid planes, the induced orientation on the corresponding
end of the moduli space changes sign. As a consequence, the boundary
contributions coming from these two-level buildings cancel pairwise.
Thus these configurations become invisible in the definition of the
differential, and the proof that $\partial^2=0$ reduces to the
dynamically convex case treated by Hutchings and Nelson in
Theorem~\ref{thm:HN}.

The abstract criterion stated above (Theorem \ref{main1}) can be verified
for concrete Hamiltonian systems arising in celestial mechanics.
In particular, it applies to the regularized energy levels of the
circular planar restricted three-body problem for energies slightly above
the first Lagrange value, see \cite{LS}. The contact manifold is diffeomorphic to $S^1 \times S^2$ (or its quotient $\R P^3 \# \R P^3$ under the antipodal map) equipped with the standard tight contact form. 
In this case, the energy hypersurface is weakly convex and the unique  contractible
Reeb orbit of Conley-Zehnder index $2$ (the Lyapunov orbit) satisfies the parity condition
for holomorphic planes for suitable choices of $J$ after slightly perturbing the contact form so it becomes nondegenerate. Hence, cylindrical contact homology is
well-defined in this setting. Other concrete examples of weakly convex contact forms are found in \cite{ dePauloHryniewiczKimSalomao2024, dPSS,dePauloSalomao2019}.

The algebraic formalism of contact homology was introduced as part of Symplectic Field Theory (SFT) by Eliashberg, Givental, and Hofer~\cite{EGH00}. 
Its full analytic foundations require a general transversality theory, for which the polyfold framework developed by Hofer, Wysocki, and Zehnder~\cite{HWZ_polyfold} provides an approach. For recent surveys and further developments, see Bourgeois~\cite{Bourgeois2025}. 
Additional references concerning cylindrical contact homology and related constructions include 
\cite{EGH00, yau2004, vK08, yau2009, Momin, AM12, N, HM, Golovko, BH18, HMark, Digiosia, HMS}.

The paper is organized as follows. In Section~\ref{sec:basics}, we review
the basic notions of Reeb flows and pseudo-holomorphic curves in
symplectizations, recalling the compactness and
transversality results that will be used throughout the paper.
In Section~\ref{sec:index}, we classify all holomorphic buildings of
Fredholm index $1$ and $2$ under the weak convexity assumption.
We then determine which index-$2$ buildings may arise as boundary
points of one-parameter families of index-$2$ holomorphic cylinders.

In Section~\ref{sec:gluing}, we analyze in detail the buildings that
obstruct the definition of the differential $\partial$. The relevant
configurations have two levels: the top level consists of a curve with
one positive end and $d+1$ negative ends ($d\geq 1$), Fredholm index
$2-d$, which is a $d$-fold cover of a pair of pants with one positive
and two negative punctures. The bottom level contains a trivial cylinder
together with $d$ rigid holomorphic planes asymptotic to a common
index-$2$ Reeb orbit. We study the associated obstruction bundle and
analyze the zeros of its canonical section in order to determine when
such buildings can be glued to genuine index-$2$ cylinders.

Finally, in Section~\ref{sec:homologia}, we complete the proof of
Theorem~\ref{main1} and establish that the differential satisfies
$\partial^2=0$. Theorem \ref{main2} is also proved.

\section{Basics}
\label{sec:basics}

\subsection{Reeb flows} A contact form on a closed three-manifold $M$ is a $1$-form $\lambda$ so that $\lambda \wedge d\lambda$ never vanishes. It determines  the Reeb vector field $R$  by the relations
$d\lambda(R, \cdot) \equiv 0$ and $\lambda(R) \equiv 1.$ Its flow $\phi_t,t\in \R$, preserves $\lambda$ and is called the Reeb flow of $\lambda$.
The contact structure $\xi=\ker \lambda\subset TM$ is  equipped with the symplectic form $d\lambda|_\xi$. 

A Reeb orbit is a pair $\alpha=(x,T)$, where $x:\R \to M$ is a periodic orbit of $\phi_t$ and $T>0$ is a period of $x$. The set of Reeb orbits is denoted $\mathcal{P}(\lambda)$,  identifying those with the same image and the same period. We call $\alpha$ simple if $T>0$ is the smallest period of $x$. Otherwise,  $\alpha$ is a $k$-cover of a simple Reeb orbit $\alpha_0=(x_0,T_0)$, and  $d(\alpha)=\frac{T}{T_0}$ is the covering number of $\alpha$. We say that $\alpha$ is non-degenerate if the symplectic linear mapping
 $
  d\phi_T:\xi_{x(0)} \to \xi_{x(0)}$ does not admit $1$  as an eigenvalue.
 We assume that $\alpha$ is non-degenerate, i.e., every Reeb orbit of $\phi_t$ is non-degenerate. We call $\alpha$ elliptic if $d\phi_T$ has eigenvalues in the unit circle. It is called positive hyperbolic if $d\phi_T$ has distinct positive real eigenvalues, and negative hyperbolic if $d\phi_T$ has distinct negative real eigenvalues.

We say that $\alpha \in \mathcal{P}(\lambda)$ is {\bf bad} if it is an even cover of a negative hyperbolic orbit. If $\alpha$ is not bad, we call it {\bf good}. The respective sets of bad and good Reeb orbits are denoted by $\mathcal{P}_{\rm bad}(\lambda)$ and $\mathcal{P}_{\rm good}(\lambda)$.

\subsection{Pseudo-holomorphic curves} Let $(\Sigma,j)$ be a closed Riemann surface, $\Gamma \subset \Sigma$ a finite set of punctures, and $(W,J)$ an almost complex manifold. Let $\dot \Sigma := \Sigma \setminus \Gamma$. A differentiable map $u:\dot \Sigma \to W$ is  called a $J$-holomorphic curve if 
$
\bar{\partial}_J(u):=(du+J(u) \circ du \circ j)/2\equiv 0.
$
We say that $u$ is somewhere injective if there exists $z \in \dot \Sigma$ such that $du(z):T_z\dot \Sigma \to T_{u(z)}W$ is injective and $u^{-1}\left(u(z)\right)=\{z\}$.
We call $u$ multi-covered if there exist a compact Riemann surface $(\Sigma',j')$, a finite subset $\Gamma'\subset \Sigma'$, a somewhere injective pseudo-holomorphic curve $u' :\Sigma' \to W$, and a holomorphic branched covering $\phi:\dot \Sigma \to \Sigma'\setminus \Gamma'$ such that
$u=u' \circ \phi$ and ${\rm deg}(\phi)>1.$
The curve $u$ is called simple if it is not multi-covered, i.e, when $u$ is somewhere injective.
Every pseudo-holomorphic curve $u$ can be factored through a somewhere injective curve $u'$ and a holomorphic branched cover $\varphi$ such that $u=v \circ \varphi$, see  \cite[Appendix]{prop2}  and \cite{mcduff2004j}. The multiplicity of  $u$ is uniquely defined by
\begin{equation}
d(u):=\left\{\begin{array}{ll}
 {\rm deg}(\phi),    & \mbox{if } u \mbox{ is multi-covered},  \\
   1,  & \mbox{if } u \mbox{ is somewhere injective}.
\end{array}
\right.
\label{d2}
\end{equation}

Let $(W, \omega)$ be a symplectic manifold and $J:TW \to TW$ be an almost complex structure. We say that $J$ is $\omega$-compatible if $g_J:=\omega(\cdot ,J \cdot)$ is a Riemannian metric on $W$.  
Let $\mathcal{J}(W,\omega)$ be the set of $\omega$-compatible almost complex structures on $(W,\omega)$.

If $(M,\xi=\ker \lambda)$ is a co-oriented contact manifold, its symplectization $\R \times M$ is equipped with the symplectic structure $\Omega:=d(e^a \lambda)$, where $a$ is the $\R$-coordinate and $\lambda$ is seen as an $\R$-invariant 1-form on $\R \times M$. The symplectization $\R \times M$ admits the $\Omega$-compatible and  $\R$-invariant almost complex structures $J$ satisfying
$J \cdot \partial_a = R$, and so that $J(\xi)=\xi$ is $d\lambda$-compatible,
where $\xi$ is seen as an $\R$-invariant bundle over $\R \times M$.
The space of such $J$'s is denoted by
$\mathcal{J}(\lambda) \subset \mathcal{J}(\R \times M, d(e^a \lambda))$.

Let $J\in \J(\lambda)$ and let $\tilde u=(a,u):\dot \Sigma \to \R \times M$ be  $J$-holomorphic.  The Hofer energy of $\widetilde{u}$ is defined as
$
E(\widetilde{u})=\sup_{\phi \in \Lambda} \int_{\Sigma \backslash \Gamma} \widetilde{u}^*d\lambda_{\phi},
$
where $\Lambda=\left\{\phi \in C^{\infty}\left(\R,[0,1]\right); \phi' \geq 0 \right\}$ and $\lambda_{\phi}$ is the $1$-form on $\R \times M$ given by
$\lambda_{\phi}(a,u)=\phi(a)\lambda(u).$ If
$
0<E(\tilde u) < +\infty,
$
then $\tilde u$ is called a finite energy curve. In particular, $\tilde u$ is non-constant.

\begin{theorem}[Hofer \cite{Hofer1993, prop1}]\label{teo:Hofer}
Let $\lambda$ be nondegenerate, $J\in \J(\lambda)$ and  $\widetilde{u}=(a,u):([0,+\infty) \times \R / \Z,i) \to (\R \times M,J)$ be a finite energy curve. Assume that $a$ is not bounded. Then there exists $\epsilon\in \{-1,1\}$  and $\alpha=(x,T) \in \P(\lambda)$, such that $u(s, \cdot) \to x(\epsilon T\cdot)$ in $C^\infty$ as $s \to \infty$. Moreover, one of the following alternatives hold
\begin{itemize}
    \item[(i)] if $\epsilon=1$, then $a(s,t) \to +\infty$ as $s \to \infty$ uniformly in $t$.
    \item[(ii)] if $\epsilon=-1$, then $a(s,t) \to -\infty$ as $s \to \infty$ uniformly in $t$.
\end{itemize}
\end{theorem}

Let $\tilde u=(a,u):(\Sigma \setminus \Gamma,j) \to (\R \times M, J)$ be a finite energy curve. If $a$ is bounded near $z^* \in \Gamma$, then it is possible to show that $z^*$ is removable. Otherwise, we consider a punctured neighborhood of $z^* \in \Gamma$ which is conformal to $([0,+\infty) \times \R / \Z,i)$. We call the puncture $z^*$  positive or  negative according to the sign of $\epsilon$ in Theorem \ref{teo:Hofer}. The corresponding Reeb orbits are called the asymptotic limits of $\tilde u$ at $z^*$. If an asymptotic limit $\alpha \in \P(\lambda)$ at $z^*\in \Gamma$ is non-degenerate, then it is the unique asymptotic limit at $z^*$ and the convergence of $\tilde u$ to $\alpha$ is exponential.  We may assume that $\Gamma$ has only non-removable punctures, which induces the splitting $\Gamma = \Gamma^+ \cup \Gamma^-.$ If $\Sigma$ is closed, then the maximum principle applied to the function $a$ implies that $\Gamma^+ \neq \emptyset.$ Our standing assumption that $\lambda$ is non-degenerate implies that $\tilde u$ is asymptotic to a unique Reeb orbit at each puncture.

We say that two $J$-holomorphic curves $\tilde u:(\Sigma \setminus \Gamma,j) \to (\R \times M, J)$ and $\tilde v:(\Sigma'\setminus \Gamma',j') \to \R \times M$ are equivalent if there exists a bi-holomorphisms $\varphi:(\Sigma\setminus \Gamma,j) \to (\Sigma'\setminus \Gamma',j')$ such that $\tilde v = \tilde u \circ \varphi$. Let $\alpha_1, \ldots, \alpha_k, \beta_1, \ldots, \beta_l\in \P(\lambda)$, and let $g$ be a non-negative integer. We denote by
$\M^J_g(\alpha_1,\ldots,\alpha_k;\beta_1,\ldots,\beta_l)$
the space of equivalent classes of finite energy $J$-holomorphic curves of genus $g$, which have $k$ positive ends asymptotic to the Reeb orbits $\alpha_1,\ldots,\alpha_k$ and $l$ negative ends asymptotic to Reeb orbits $\beta_1 ,\ldots,\beta_l$.
Finally, we denote by
$\M^J_g(\alpha_1,\ldots,\alpha_k;\beta_1,\ldots,\beta_l)/\R$
the quotient of $\M^J_g(\alpha_1,\ldots,\alpha_k;\beta_1,\ldots,\beta_l)$ under the $\R$-action $\rho \cdot (a,u) = (a+\rho,u), \forall \rho\in\R$.

\subsection{The Conley-Zehnder index}
Let $(M,\xi = \ker \lambda)$ be a closed contact $3$-manifold, where $\lambda$ is nondegenerate. 
Given $\alpha=(x,T)\in \P(\lambda)$, let $x_T:=x(T\cdot):\R \to M$. Let $\mu_{\rm CZ}^\tau(\alpha)$ be the Conley-Zehnder index of $\alpha$ with respect to a symplectic trivialization $\tau:x_T^*\xi \to \R/ \Z \times \C$. It is defined as follows. The linearized flow along $\alpha$ and the trivialization $\tau$ determine a path of symplectic matrices $\Psi_t = \tau \circ d\phi_{Tt} \circ \tau^{-1}\in {\rm Sp}(1),$ $t\in [0,1],$  starting from the identity and ending at a matrix which does not admit $1$ as an eigenvalue.  A curve $v_t = \Psi_t \cdot v_0,t\in[0,1],$ starting from $v_0\in \C \setminus 0,$ writes as $v_t = r(t) e^{i \theta(t)},$ for continuous $r(t)>0$ and $\theta(t)\in \R.$ Varying $v_0$ along $S^1 \subset \C$, the set of all possible values for $\frac{\theta(1)-\theta(0)}{\pi}$ is a compact interval $K \subset \R$ of length  $<1$. Because $\alpha$ is nondegenerate, we have
$ \partial K \cap \Z = \emptyset.$
We define
$$
\mu_{\rm CZ}^\tau(\alpha) := \left\{\begin{aligned} & 2k, & \mbox{ if }  K \mbox{ contains } k\in \Z, &\\ & 2k+1, & \mbox{ if } K\subset (k,k+1) \mbox{ for some } k\in \Z.  \end{aligned} \right.
$$

A trivialization $\tau$ of $x_T^*\xi$ induces a trivialization $\tau^d$ along the iterates $\alpha^d=(x,dT), \forall d\in \N$. In this case, we  simply denote by $\mu_{\rm CZ}^\tau(\alpha^d)$ the Conley-Zehnder index of $\alpha^d$ with respect to $\tau^d$. 

The rotation number of $\alpha$ with respect to $\tau$  is then defined as
$$
\theta^\tau(\alpha) : =  \lim_{d \to \infty} \frac{\mu_{\rm CZ}^\tau(\alpha^d)}{2d}.
$$
One can check that
$
\mu_{\rm CZ}^\tau(\alpha^d) = \lfloor d \theta^\tau(\alpha) \rfloor + \lceil d\theta^\tau(\alpha) \rceil, \forall d. 
$
If $\alpha\in \P(\lambda)$ is  hyperbolic, then 
$
\mu_{CZ}^{\tau}(\alpha^d)=d\cdot\mu_{CZ}^{\tau}(\alpha), \forall d\geq 1.
$
If $\alpha\in \P(\lambda)$ is elliptic, then $\theta^\tau(\alpha)\in \R \setminus \Q$ and thus

\begin{equation}\label{eq:CZ_ellpitic}
\begin{aligned}
\CZ(\alpha^d) &= 2\lfloor d \cdot \theta^\tau(\alpha) \rfloor + 1 \\
&\geq  2d\lfloor \theta^\tau(\alpha) \rfloor + 1 \\
&= d\left(2\lfloor \theta^\tau(\alpha) \rfloor + 1\right)-d+1 \\
&= d \cdot \CZ(\alpha)-(d-1).
\end{aligned}
\end{equation}
We shall assume that the first Chern class of $\xi \to M$ vanishes on $\pi_2(M)$, that is 
$
c_1(\xi)|_{\pi_2(M)} \equiv 0.$

If $\alpha=(x,T)\in \P(\lambda)$ is  contractible, then we denote by $\mu_{\rm CZ}(\alpha) := \mu_{\rm CZ}^{\tau_*}(\alpha),$
  the Conley-Zehnder index of $\alpha$ with respect to a symplectic trivialization $\tau_*$  of $x_T^*\xi$ that extends over  a capping disk $u$ for $\alpha$, i.e. if $u:\mathbb{D} \to M$ is continuous and satisfies  $u(e^{i2\pi t}) = x(Tt), \forall t$, then $\tau_*$ continuously extends to a symplectic trivialization of $u^*\xi.$ By assumption,  $\mu_{\rm CZ}(\alpha)$ does not depend on the capping disk $u$. We denote 
\begin{equation*}
    \mathcal{P}^\tau_{{\rm good},i}(\lambda):=\left\{\alpha \in \mathcal{P}^{\tau}_{\rm good}(\lambda): \mu_{\rm CZ}^{\tau}(\alpha)=i-1\right\},
\end{equation*}
where $\tau$ is a given trivialization.

Let $\tilde u=(a,u):\dot \Sigma \to \R \times M$ be a pseudo-holomorphic curve with positive ends in $\alpha_1,...,\alpha_n$ and negative ends in $\beta_1,...,\beta_k$, and let  $\tau$ be a set of trivializations of $\xi$ along $\alpha_1,...,\alpha_n,\beta_1,...,\beta_k$. The  Fredholm index of $\tilde u$ is defined as
\begin{equation}
{\rm ind}(\tilde u):=-\chi(u)+2c_{\tau}(u)+\sum_{i=1}^k \mu_{CZ}^{\tau}(\alpha_i)-\sum_{i= 1}^l \mu_{CZ}^{\tau}(\beta_i),
\end{equation}
where \begin{equation}
   \chi(u):=\chi(\dot \Sigma) =2-2g-\# \Gamma= 2-2g-k-l,
   \label{eq:chi(u)}
\end{equation}
 and $c_{\tau}(u)$ is the relative first Chern class of  $u^*\xi$ with respect to $\tau$. Notice that if $\tau$ continuously extends as a trivialization of $u^*\xi$, then $c_\tau(u)=0.$

We denote by
$\M^J_{g,p}(\alpha_1,\ldots,\alpha_k;\beta_1,\ldots,\beta_l) \subset \M^J_g(\alpha_1,\ldots,\alpha_k;\beta_1,\ldots,\beta_l)$
the set of all $J$-holomor\-phic curves $\tilde u\in \M^J_g(\alpha_1,\ldots,\alpha_k;\beta_1,\ldots,\beta_l)$ such that $\ind(\tilde u)=p$. Its quotient under the action induced by $\R$-translations is denoted  by
$
 \M^J_{g,p}(\alpha_1,\ldots,\alpha_k;\beta_1,\ldots,\beta_l)/\R.
 $

\subsection{The asymptotic operator}
Let $(M,\xi=\ker \lambda)$ be a closed contact three-manifold with nondegenerate contact form$\lambda$, $J \in \mathcal{J}(\lambda)$ and $\alpha$ a  $T$-periodic Reeb orbit. Let $Y:\R / \Z \to \alpha_T^*\xi$ be  smooth, where $\alpha_T := \alpha(T\cdot)$.
The Lie derivative  of $Y$ in the direction of $\dot \alpha_T= T\cdot  R_\lambda \circ
\alpha_T$ is given by
$$(\mathcal{L}_{\dot \alpha_T}Y)(t)=\frac{d}{ds}\Big|_{s=0}d\psi_{-sT}\big(\alpha_T(t+s )\big) \cdot Y(t+s), \quad \forall t\in \R/\Z,$$
where $\psi_t$ is the flow of $R_\lambda$. 

The {\bf asymptotic operator}  $A_\alpha:H^1(\alpha_T^*\xi) \subset L^2(\alpha_T^*\xi) \to L^2(\alpha_T^*\xi)$
associated with $\alpha$ and $J$ is defined by
$$A_\alpha Y:=-J\cdot \mathcal{L}_{\dot \alpha_T}Y, \quad \forall Y \in H^1(\alpha_T^*\xi).$$
Although $A_\alpha$ depends on $J$, we omit $J$ in the notation for simplicity. The asymptotic operator $A_\alpha$ is an unbounded operator between suitable Hilbert spaces, self-adjoint with respect to the inner product
$$\left< Y_1,Y_2 \right>_J:=\int_{\R / Z} d\lambda_{\alpha(t)}(Y_1(t),J_{\alpha(t)}Y_2(t) ) \, dt, \quad \forall Y_1,Y_2 \in H^1(\alpha_T^* \xi).$$
The spectrum ${\rm spec}(A_\alpha)$ of $A_\alpha$ consists of real eigenvalues that accumulate at $\pm \infty$. The Reeb orbit $\alpha$ is  non-degenerate  if and only if  $\ker A_\alpha$ is trivial.  
Using a trivialization $\tau:\alpha_T^*\xi \to \R / \Z \times \R^2$ satisfying $J \equiv i$ and $d\lambda \equiv dx \wedge dy$, the asymptotic operator $A_\alpha$ takes the form
\begin{equation}
A_{\alpha} = -i\frac{d}{dt}-S_{\infty}, 
\end{equation}
where $S_{\infty}:\R / \Z \to {\rm Sym}(2,\R)$ is a loop of $2\times 2$ real symmetric matrices. Hence, a non-trivial $\mu$-eigenvector $y:\R / \Z \to \R^2 $ of $A_{\alpha}$ never vanishes and thus has a well-defined winding number $\wind(\mu) \in \Z,$ not depending on the eigenvector.  

\begin{theorem}[Hofer-Wysocki-Zehnder \cite{prop2}] \label{windhwz} The following assertions hold:
\begin{itemize}
\item[(i)] If $y_1$ and $y_2$ are (non-trivial) eigenvectors of $A_{\alpha}$ associated with the eigenvalue $\mu\in \R$, then their winding numbers are the same. 

\item[(ii)] Let  $\mu_1 < \mu_2$ be eigenvalues of $A_\alpha$, and let $y_1$ and $y_2$ be (non-trivial) eigensections associated with $\mu_1$ and $\mu_2$, respectively. Then $\wind(\mu_1) \leq \wind(\mu_2)$. Moreover, if $\wind(\mu_1)=\wind(\mu_2)$, then $y_1$ and $y_2$ are pointwise linearly independent.

\item[(iii)] Given $k \in \Z$, there are precisely two eigenvalues $\mu_1$ and $\mu_2$ of $A_{\alpha}$, counting multiplicities, such that $k=\wind(\mu_1)=\wind(\mu_2).$

\end{itemize}
\end{theorem}

It is possible to show that
$
\mu_{\rm CZ}^\tau(\alpha) = \wind_{<0}(A_\alpha) + \wind_{\geq 0}(A_\alpha),
$
where
$$
\begin{aligned}
\wind_{<0}(A_\alpha) & := \max \{\wind(\mu): \mu \in {\rm spec}(A_\alpha) \cap (-\infty,0)\},\\
\wind_{\geq 0}(A_\alpha) & := \min \{\wind(\mu): \mu \in {\rm spec}(A_\alpha) \cap [0,\infty)\}.
\end{aligned}
$$

Let $\tilde u=(a,u):(\Sigma,j) \to (\R \times M,J)$ be a $J$-holomorphic curve such that, at a positive puncture $z \in \Gamma_+$, $\tilde u$ is asymptotic to a nondegenerate and $T$-periodic Reeb orbit $\alpha$. Assume that $\alpha=\alpha_0^m$, where $\alpha_0$ is a simple $T_0$-periodic Reeb orbit ($T=mT_0)$. We can use polar coordinates $s+it \in [s_0,+\infty) \times \R / \Z$ in a punctured neighborhood of $z \in \Sigma$ and a tubular neighborhood of $\alpha_0(\R) \subset M$, with coordinates $(\theta,x,y)\in \R / T_0\Z \times B_\delta(0), \delta>0$ small, such that $\tilde u$ has the form
$$\tilde u(s,t)=(a(s,t),\theta(s,t),z(s,t)), \quad \forall (s,t) \in [s_0,+\infty) \times \R / \Z,$$
where $\theta(s,t+1)=\theta(s,t) +mT_0 \in \R,$ and $z(s,t) \in B_\delta(0) \subset \R^2 , \forall (s,t)$.
The asymptotic behavior of $\tilde u$ near $z$ is described by the following theorem. 

\begin{theorem}[Hofer-Wysocki-Zehnder \cite{prop1}] \label{reprass}
Let $\widetilde{u}=(a,\theta,z):[s_0,+\infty)\times \R/\Z \to \R^4$ be  as above. Assume further that $\widetilde{u}$ is not a trivial cylinder. Then, there are constants $c_1,c_2 \in \R$, $d>0$ and $M_{ij}>0$ such that
$
|\partial^{(i,j)}[a(s,t)-Ts-c_1]| \leq M_{ij}e^{-ds}$ and $
|\partial^{(i,j)}[\theta(s,t)-mt-c_2]| \leq M_{ij}e^{-ds}
$
for every $s \geq s_0$, $t \in \R$ and every multi-index $(i,j) \in \N \times \N$. Furthermore,
$
z(s,t)=e^{\mu s}(\varphi(t)+h(s,t)),
$
where $\varphi$ is an eigenvector of  $A_{\alpha}$ associated with an eigenvalue $\mu<0$, and $h(s,t)$ satisfies
$\partial^{(i,j)}h(s,t) \to 0$
uniformly in $t \in \R$, for all $(i,j) \in \N \times \N$. 
\end{theorem}

The eigenvalue $\mu<0$ and a $\mu$-eigensection $\varphi$ as in Theorem \ref{reprass} are called the leading eigenvalue and a leading eigensection of $\tilde u$ near $z\in \Gamma_+$, respectively. If $z\in \Gamma_-$, then an analogous statement describing $\tilde u$ near $z$ holds for a leading eigenvalue $\mu>0$.

Fixing a symplectic trivialization $\tau$ of $u^*\xi$, denote also by $\tau$ the induced trivialization of $\xi$ along the nondegenerate asymptotic limits $\alpha_z, z\in \Gamma,$ of $\tilde u$. Let $\mu_z$ be the leading eigenvalue  of $A_{\alpha_z}$ as in Theorem \ref{reprass}. Define
\begin{equation}
\wind_{\infty}(u;z):=\wind(\mu_z),
\label{eq:windinfty}
\end{equation}
and
$$
{\rm wind}(\tilde u):= \sum_{z\in \Gamma_+} \wind(\tilde u,z) - \sum_{z\in \Gamma_-} \wind(\tilde u,z).
$$
It can be shown that ${\rm wind}(\tilde u)$ does not depend on $\tau$, see \cite{prop2}.

\subsection{The linearization of the Cauchy-Riemann Operator}\label{sec_linear}

Let $\lambda$ be a nondegenerate contact form on a closed three-manifold $M$, and let $J \in \mathcal{J}(\lambda)$. Let $(\Sigma,j)$ be the Riemann sphere and $\Gamma \subset \Sigma$ a finite set. Let $u:(\dot{\Sigma},j) \to (W:=\R \times M, J)$ be an immersed $J$-holomorphic curve satisfying
\begin{itemize}
     \item[(i)] $\#\Gamma \geq 3;$
     \item[(ii)] $u$ is an immersion;
    \item[(iii)] $\int_{\dot \Sigma} u^* d\lambda >0$. 
   
\end{itemize}

Let 
$
F:=\overline{\rm Hom}_{\C}\left(T\dot{\Sigma},u^*TW\right) \to \dot{\Sigma},
$ 
be the vector bundle over $\dot \Sigma$ consisting of all complex anti-linear maps $(T\dot{\Sigma},j) \to (u^*TW,J)$. Choose an $\R$-invariant symmetric connection $\nabla$ on $\R \times M$, and
consider the Cauchy-Riemann operator
$D_u: \Gamma(u^*TW) \to \Gamma(F)$
 defined as
\begin{equation}\label{eq:opDu}
     D_u \eta:=\nabla \eta + J(u) \circ \nabla \eta \circ j + (\nabla_\eta J)\circ du \circ j, \quad \forall \eta \in \Gamma(u^*TW).  
\end{equation}

Let $T_u,N_u \subset u^*TW$ be the tangent and normal bundles over $u$ equipped with the complex structure induced by $J$, so that
$
u^*TW = T_u \oplus N_u.
$
By Theorem \ref{reprass}, it is possible to identify $N_u\equiv u^*\xi$ near the punctures in $\Gamma$.
Following \cite{W10}, the normal Cauchy-Riemann operator at $u$
$
D_u^N:W^{1,p}(N_u)  \to L^p(\overline{\rm Hom}_{\C}(T\dot{\Sigma},N_u))$
is regarded as a restriction and projection of $D_u$ to $N_u$.

Let $z^* \in \Gamma$ be a puncture where $u$ is asymptotic to $\alpha\in \P(\lambda)$. In a punctured neighborhood of $z^*$ in $\Sigma$, with cylindrical coordinates $(s,t)\in Z_+ := [0,+\infty)\times \R / \Z$, and a suitable trivialization of $N_u$, the operator $D_u^N$  writes as
$D_u^N = \bar{\partial}+S,$
where $\bar{\partial}=\partial_s+J_0\cdot \partial_t$ and $S \in C^{\infty}(Z_{\pm},{\rm End}(\R^{2}))$. Moreover,  $S(s,\cdot)$ converges in $C^\infty(\R / \Z)$ to a  loop $S_\infty$ of symmetric matrices.

The assumptions (i), (ii), and (iii) above imply that 
$$
\begin{aligned}
    \ind(D_u^N) & =\ind(u),\\
    \dim \ker (D_u^N) & = \dim \ker (D\bar{\partial}_J(u,j)), \\
    \dim \coker (D_u^N) & = \dim \coker (D\bar{\partial}_J(u,j)),
\end{aligned}
$$
where $\bar \partial_J:\mathcal{B} \to \mathcal{E}:(j,u) \mapsto du + J \circ du \circ j$ is a smooth section of a certain Banach space bundle,  and $D\bar \partial_J: T\mathcal{B} \to T \mathcal{E}$ is its linearization. See  \cite[Theorem 3]{W10}.

Let
$
F^N:=\overline{\rm Hom}_{\C}(T\dot{\Sigma},N_u) \to \dot{\Sigma}.
$ 
The formal adjoint of $D_u^N$ is the linear map
$(D_u^{N})^*: \Gamma(F^N) \to \Gamma(N_u)$
satisfying
$$
\langle \zeta, D_u\eta\rangle_{L^2(F^N)}=\langle (D_u^{N})^*\zeta, \eta\rangle_{L^2(N_u)}, \quad \forall \eta \in C_0^{\infty}(N_u), \quad \forall \zeta \in C_0^{\infty}(F^N).
$$
For suitable choices of trivializations of $N_u$ and holomorphic coordinates $(s,t)\in \dot \Sigma$, $(D_u^N)^*$ expresses as
$(D_u^{N})^*=-\partial+S^T,$
where $\partial:=\partial_s-J_0 \cdot \partial_t$, and $S^T \in {\rm End}(\R^2)$, see \cite{W16} for more details. 
The operator $(D_u^{N})^*$ is also Fredholm when defined as 
$
(D_u^{N})^*: W^{k,p}(F^N) \to W^{k-1 ,p}(N_u).
$
Its kernel is contained in $W^{m,q}(F^N)$ for all $m \in \N$ and $q \in (1, \infty),$ and
$W^{k-1,p}(F^N)={\rm Im} (D_u^N) \oplus \ker (D^{N}_u)^*.$
If $$C:= \left(\begin{array}{cc}
1 & 0\\
0 & -1
\end{array}\right)\in {\rm End}(\R^{2}),$$ then
$C^{-1}(D_u^N)^*C=-(\bar{\partial} + \bar{S}),$
where $\bar{S}(s,t):=-CS(s,t)^TC$. 

Recall that the asymptotic operator over $\alpha$ writes as $A_{\alpha}=-J_0\frac{d}{dt}-S_ {\infty}$, where $J_0$ corresponds the multiplication by $i$, $S_\infty=\lim_{s\to \infty} S(s,\cdot)$ is symmetric, and $(s,t) \in Z_+$ are suitable holomorphic coordinates near $z^*\in \Gamma.$  We have
$\lim_{s \to +\infty}\bar S(s,\cdot) = \bar{S}_{\infty},$
  where $\bar{S}_{\infty}(t)=-CS_{\infty}(t)C$. Thus the limiting operator  $\bar{A}_{\alpha}:=-J_0\frac{d}{dt}-\bar{S}_{\infty}$ conjugates with $-A_\alpha$, that is
\begin{equation}
C^{-1}\bar{A}_{\alpha}C=-A_{\alpha}.
\label{eq:cac1}
\end{equation}
In particular, ${\rm spec}(\bar A_\alpha) = - {\rm spec}(A_\alpha)$. 
The following proposition is straightforward. 

\begin{prop} \label{AbarA} 
Let ${\rm Aut}(\mu, A_{\alpha})$ be the linear space generated by the $\mu$-eigensections of $A_\alpha$, $\mu \in {\rm spec}(A_\alpha)$. Then $\varphi \in {\rm Aut}(\mu, A_ {\alpha})$ if and only if $C\varphi \in {\rm Aut}(-\mu, \bar{A}_{\alpha})$. 
\end{prop}

The asymptotic behavior of a section $\sigma$ in the kernel of $(D_u^N)^*$ near a puncture is described as in Theorem \ref{reprass}. In fact, in polar holomorphic coordinates $(s,t) \in [0,+\infty) \times\R/ \Z$ or $(s,t)\in (-\infty,0] \times \R / Z$ near a positive or a negative puncture, respectively, we evaluate $\sigma(\partial_s)$ and see $\sigma$ as a section of the normal bundle $N_u$, which is identified with the contact structure near the ends. 

\begin{prop} \label{kerDu}
Let $\sigma\in \ker (D_u^{N})^*$. The following assertions hold.
\begin{itemize}

\item[(i)] in cylindrical coordinates $(s,t) \in (-\infty,0]\times \R / \Z$ near a negative puncture
\begin{equation}
\sigma(s,t)=e^{-\mu s}(\varphi(t)+h(s,t))
\label{eq:sigmaeq1}
\end{equation}
for some leading eigenvalue $\mu<0$  of $A_\alpha$, and a leading $\mu$-eigensection $\varphi$. Here, $h(s,t) \to 0$ in $C^\infty$ as $s \to -\infty$.

\item[(ii)] in cylindrical coordinates $(s,t) \in [0,+\infty)\times \R / \Z$ near a positive puncture
\begin{equation}
\sigma(s,t)=e^{-\mu s}(\varphi(t)+h(s,t))
\label{eq:sigmaeq2}
\end{equation}
for some leading eigenvalue $\mu>0$ of $A_\alpha$, and a leading $\mu$-eigensection $\varphi$. Here, $h(s,t) \to 0$ in $C^\infty$ as $s \to +\infty$.
\end{itemize}
\end{prop}

\begin{definition}
Let $\sigma \in \ker (D_u^{N})^* \setminus \{0\}$. For every $z \in \Gamma$, we define
 $\wind(\sigma;z):=\wind(\mu_z),$ where $\mu_z$ is the leading eigenvalue in the asymptotic representation of $\sigma$ near $z$ as in Proposition \ref{kerDu}. Here, $\tau$ is a symplectic trivialization of the contact structure along the asymptotic limit at $z$ induced by a fixed trivialization of the normal bundle $N_u$. Let
     $${\text wind}_{\infty}(\sigma):=\sum_{z \in \Gamma_{+}}\wind(\sigma;z)-\sum_{z \in \Gamma_{-}} \wind(\sigma;z).$$  Then ${\rm wind}_\infty(\sigma)$ does not depend on $\tau$.
     Finally, denote by $\sharp \sigma$ the algebraic sum of the zeros of $\sigma$. Notice that $\#\sigma$ is finite since $\sigma$ does not vanish near the punctures and each zero of $\sigma$ is isolated.
    
\end{definition}

\begin{lemma}[{\cite[Proposition 5.6]{prop2}}] \label{zeroswindcf}
Let $\sigma \in \ker (D_u^N)^*$. If $\sigma$ does not vanish identically, then $\sharp \sigma = {\rm wind}_{\infty}(\sigma)+\chi(\Sigma)-\sharp \Gamma\leq0$.
\end{lemma}

We use Lemma \ref{zeroswindcf} to prove the following proposition.

\begin{prop} \label{kerDu2}
Suppose $u \in \M^J_0(\alpha; \beta_1,\beta_2,\cdots,\beta_2)$ is an embedded $J$-holomorphic curve with a positive puncture $z_0$ asymptotic to an elliptic Reeb orbit $\alpha$, a negative puncture $z_1$ asymptotic to an elliptic Reeb orbit $\beta_1$, and $d$ negative punctures $z_2, \ldots, z_{d+1}$ all asymptotic to $\beta_2$, where $d > 1$. Suppose further that $\mu_{CZ}^{\tau_N}(\alpha)-\mu_{CZ}^{\tau_N}(\beta_1)=2$, and that $\beta_2$ is a contractible Reeb orbit with $\mu_{CZ}^{\tau_N}(\beta_2)=0$. Here, $\tau_N$ is a symplectic trivialization of the contact structure along the asymptotic limits $\alpha,\beta_1,\beta_2,\ldots,\beta_2$ induced by a trivialization of the normal bundle $N_u$. Then, $$ 2\leq {\text wind}_{\infty}(\sigma) \leq d \quad \mbox{ and } \quad 2-d \leq \#\sigma \leq 0. $$
Furthermore, at least two punctures $z_{i_1}, z_{i_2}$ in $\{z_2,\ldots, z_{d+1}\}$ satisfy $$
{\text wind}^{\tau_N}(\sigma;z_{i_1})={\text wind}^{\tau_N}(\sigma;z_{i_2})=0.
$$
\end{prop}

\begin{proof}
By Lemma \ref{zeroswindcf}, 
${\rm wind}_{\infty}(\sigma) = \sharp \sigma -\chi(\Sigma)+\sharp \Gamma$
and $\sharp \sigma \leq 0$. Hence,
${\rm wind}_{\infty}(\sigma) \leq -2+(d+2) = d.
$
 On the other hand,
${\rm wind}_{\infty}(\sigma)={\rm wind}^{\tau_N}(\sigma;z_0)-\sum_{i=1}^{d+1} {\rm wind}^{\tau_N}(\sigma;z_i). 
 $
Since $\alpha$ and $\beta_1$ are elliptic, we have
$
\mu_{CZ}^{\tau_N}(\alpha)=2 {\rm wind}^{\tau_N}_{\geq 0}(A_{\alpha})-1$ and $\mu_{CZ}^{\tau_N}(\beta_1) = 2{\rm wind}^{\tau_N}_{<0}(A_{\beta_1})+1.
$
Hence
$$
2= \mu_{CZ}^{\tau_N}(\alpha)-\mu_{CZ}^{\tau_N}(\beta_1)=2({\rm wind}^{\tau_N}_{\geq 0}(A_{\alpha})-{\rm wind}^{\tau_N}_{<0}(A_{\beta_1}))-2,
$$
and thus
$2= {\rm wind}^{\tau_N}_{\geq 0}(A_{\alpha})-{\rm wind}^{\tau_N}_{<0}(A_{\beta_1}) \leq {\rm wind}^{\tau_N}(\sigma; z_0) -{\rm wind}^{\tau_N}(\sigma; z_1).$

We also have
$$
{\rm wind}^{\tau_N}(\sigma;z_i) \leq {\rm wind}^{\tau_N}_{<0}(A_{\beta_2})=\frac{\mu_{CZ}^{\tau_N}(\beta_2)}{2}=0, \quad \forall i.
$$
Using the inequalities above, we conclude that
$$
{\rm wind}_{\infty}(\sigma) = {\rm wind}^{\tau_N}(\sigma;z_0) - {\rm wind}^{\tau_N}(\sigma;z_1) - \sum_{i=2}^{d+1} {\rm wind}^{\tau_N}(\sigma;z_i) \geq 2.
$$
The desired inequalities for ${\rm wind}_\infty(\sigma)$ and $\#\sigma$ follow.  

Finally, if $d-1$ punctures $z_{i_1}, \ldots, z_{i_{d-1}}$ in the set $\{z_2, \ldots, z_{d+1}\}$ satisfy ${\rm wind}^{\tau_N}(\sigma;z_{i_j})<0$, then
$
{\rm wind}_{\infty}(\sigma) \geq 2+d-1=d+1,
$
a contradiction.  
\end{proof}

\section{Holomorphic Buildings} \label{sec:index}
In this section, we list all possible holomorphic buildings of Fredholm indices 1 and 2 associated with a weakly convex nondegenerate contact form and a generic $J$. In addition, we decide which of these buildings may appear as the limit of a sequence of an index-$2$ family of holomorphic cylinders. 

\subsection{Buildings of index 1 and 2}
We start recalling some useful results from \cite{HN,HT07,W16}.

\begin{lemma}[Hutchings-Taubes \cite{HT07}] \label{1.7} 
Let $(M,\lambda)$ be a contact $3$-manifold, $J \in \mathcal{J}(\lambda)$, and $v:\Sigma \setminus \Gamma \to \R \times M$ a $J$-holomorphic curve. If $v$  multiply covers a trivial cylinder, then $\ind(v)\geq 0$.
\end{lemma}

\begin{lemma}[Hutchings-Nelson \cite{HN}] \label{2.2} 
Let $(M,\lambda)$ be a 3-contact manifold, $J \in \mathcal{J}(\lambda)$, and $v:\Sigma \setminus \Gamma \to \R \times M$ a genus zero $J$-holomorphic curve, with one positive puncture, and an arbitrary number of negative punctures. If $\mathfrak{v}$ is a {\it somewhere injective} curve which is covered by $v$  with multiplicity $d$, and if $b$ is the number of branching points of that cover, counted with multiplicity, then
$$\ind(v) \geq d \cdot \ind(\mathfrak{v})+2(1-d+b).$$
\end{lemma}

\begin{theorem}[Wendl \cite{W16}] \label{thmreg}
Let $(M,\lambda)$ be a contact $3$-manifold, $g$ a non-negative integer and $\alpha_1, \ldots, \alpha_k, \beta_1,\ldots,\beta_l$ non-degenerate Reeb orbits. Then, there is a dense subset
$
\mathcal{J}_{\rm reg}(\lambda) \subset \mathcal{J}(\lambda)
$
such that for each $J \in \mathcal{J}_{\rm reg}(\lambda)$ the  {\it somewhere injective} curves $u \in \M^J_g(\alpha_1, \ldots, \alpha_k; \beta_1,\ldots,\beta_l)$ are Fredholm-regular. Furthermore, these curves form an open subset of $\M^J_g=\M^J_g(\alpha_1, \ldots, \alpha_k; \beta_1,\ldots,\beta_l),$ which is a manifold of dimension 
$${\rm dim}(\M^J_g)=-(2-2g-k-l)+2c_{\tau}(u)+\sum_{i=1}^k \mu_{CZ}^{\tau}(\alpha_i)-\sum_{i=1}^l \mu_{CZ}^{\tau}(\beta_i)\geq 0.$$
\end{theorem}

\begin{lemma}[Hutchings-Nelson \cite{HN}] \label{2.3} 
Let $(M,\lambda)$ be a 3-contact manifold, $J \in \mathcal{J}_{\rm reg}(\lambda)$, and $v:\Sigma \setminus \Gamma \to \R \times M$ a $J$-holomorphic curve with a positive end and $n$ negative ends. Suppose $v$ covers a non-trivial somewhere injective cylinder. Then $\ind(v) \geq n.$
\end{lemma}

\begin{lemma}[Hutchings-Nelson \cite{HN}] \label{2.5} 
Let $(M,\lambda)$ be a 3-contact manifold, $J \in \mathcal{J}_{\rm reg}(\lambda)$ and $v:\Sigma \backslash \Gamma \to \R \times M$ a    non-trivial $J$-holomorphic cylinder. Let $\mathfrak{v}$ be the  {\it somewhere injective} cylinder covered by $v$. Then $1 \leq \ind(\mathfrak{v}) \leq \ind(v).$
\end{lemma} 

Given the above results, we are ready to prove the following proposition which refines Lemma 2.4 in \cite{HN}.

\begin{prop} \label{A} 
Let $(M, \lambda)$ be a contact $3$-manifold and $J\in \mathcal{J}_{\rm reg}(\lambda)$. Let $v:\Sigma \setminus \Gamma \to \R \times M$ be a genus zero finite energy $J$-holomorphic curve, with one positive puncture  and $n>1$ negative punctures. Assume that $v$ is not the  cover of a trivial cylinder. Then
\begin{equation}
\ind(v) \geq 3-n.
\label{eq:3-n}
\end{equation}
Furthermore, if equality holds in \eqref{eq:3-n}, then $v$  multiply covers  a somewhere injective curve $\mathfrak{v}$ of index 1 with one positive puncture and two negative punctures. Moreover, there are no branching points, $v$ has $n=d+1$ negative punctures and $\ind(v)=2-d$, where $d=d(v)$ is the multiplicity of the cover  $v\to \mathfrak{v}$.
\end{prop}

\begin{proof}
Let $\mathfrak{v}$ be the somewhere injective curve covered by $v$. Then $\mathfrak{v}$ has one positive puncture and $k>0$ negative punctures. Let $d$ be the covering multiplicity of $v$ and $b$ the number of branching points of $v$, counted with multiplicity.

If $\mathfrak{v}$ is a non-trivial cylinder, then it follows from Lemma \ref{2.3} that $\ind(v) \geq n$. Furthermore, since $n>1$, we obtain
\begin{equation}\label{eq:ntc}
\ind(v)+n \geq 2n \geq 4,
\end{equation}
and \eqref{eq:3-n} follows in this case.

Now assume that the somewhere injective curve $\mathfrak{v}$ is not a cylinder, i.e., $k>1$. Since $J\in \mathcal{J}_{\rm reg}(\lambda)$, we have 
\begin{equation}\label{indg1}
\ind(\mathfrak{v}) \geq 1.
\end{equation}
By Lemma \ref{2.2}, we conclude that
$\ind(v) \geq 2-d+2b.$
The Riemann-Hurwitz Theorem (see \cite{Oort} for the compact case) tells us that
\begin{equation*}
\chi(v)=d \cdot \chi(\mathfrak{v})-b \Rightarrow 2-1-n=d(2-1-k)-b \Rightarrow n=d(k-1)+1 +b.
\end{equation*}
Since $b \geq 0$, $d \geq 1$ and $k>1$, we conclude that
\begin{equation}
  \ind(v)+n \geq 2-d+2b+d(k-1)+1+b
    = 3(b+1)+d(k-2)
    \geq 3.
  \label{eq:ind+n}
\end{equation}
Inequality \eqref{eq:3-n} is also proved in this case.
 
Now assume that equality holds in \eqref{eq:3-n}, that is $\ind(v) +n=3$. From \eqref{eq:ntc}, $\mathfrak{v}$ is not a cylinder. Inequalities in \eqref{eq:ind+n} give
\begin{eqnarray*}
3=\ind(v)+n \geq 3(b+1)+d(k-2) \geq 3 \quad \Rightarrow \quad 3b+3+d(k-2)=3
\end{eqnarray*}
that is,
$3b+d(k-2)=0.$
Since $b \geq 0$, $k \geq 2$ and $d \geq 1$ we conclude that $b=0$ and $k=2$. Using the Riemann-Hurwitz Theorem again, we obtain
$n=d(k-1)+1+b=d+1.
$
 Lemma \ref{2.2} implies that
\begin{eqnarray*}
 3-(d+1)=\ind(v) \geq d \cdot \ind(\mathfrak{v})+2(1-d)
\ \Rightarrow \ d \cdot \ind(\mathfrak{v}) \leq d 
\ \Rightarrow \  \ind(\mathfrak{v}) \leq 1.
\end{eqnarray*}
Using \eqref{indg1} we conclude that $\ind(\mathfrak{v})=1.$ This finishes the proof of the proposition.
\end{proof}

\begin{definition}[Holomorphic buildings] \label{def:bil}
Let $(W,\omega)$ be a symplectic manifold and $J:TW \to TW$ an almost complex structure. A holomorphic building of height $m$ is an m-tuple $v=(v_1,\ldots,v_m)$ such that $v_i:\Sigma_i \setminus \Gamma_i \to W, i=1,\ldots,m,$ is a $J$-holomorphic curve, where $\Sigma_i$ is a (possibly disconnected) Riemann surface.
Moreover, the asymptotic limits of $v_i$ at the negative ends  coincide with the asymptotic limits of $v_{i+1}$ at the positive ends.

Each  $v_i$ is a level of the building $v$. The positive and negative ends of the building $v$ are, respectively, the positive ends of $v_1$ and the negative ends of $v_m$. The genus of $v$ is  the genus of the surface obtained by gluing the negative ends of $v_i$ with the respective positive ends of $v_{i+1}$, $i=1,\ldots,m-1$. The index of $v$ is the sum of the indices of all curves in the building 
$
\ind(v) = \sum_{i=1}^m \ind(v_i).
$ Each level of $v$ has a component that is not a trivial cylinder.
 
\end{definition}


Generalizations of Propositions 2.7 and 2.8 of \cite{HN} in the weakly convex case are given by the following theorems. We start considering buildings with one positive end and no negative ends.

\begin{prop} \label{B} 
Let $\lambda$ be a nondegenerate and weakly convex contact form on $(M, \xi)$. Let $J\in \mathcal{J}_{\rm reg}(\lambda)$ and let $v=(v_1, \ldots, v_m)$ be a genus zero $J$-holomorphic building  with a positive end and no negative ends. Then:
\begin{itemize}
\item[(i)] $\ind(v) \geq 1$.
\item[(ii)] If $\ind(v)=1$, then $v$ is a plane.
\item[(iii)] If $\ind(v)=2$, then $v$ is one of the following buildings:
\begin{enumerate}[label=(iii.\arabic*)]
   \item $v$ is a plane.
   \item \label{itm:Bc2} $v=(v_1,v_2)$, where $v_1$ is an index-$1$ non-trivial cylinder and  $v_2$ is an index-$1$ plane.
\end{enumerate}
\end{itemize}
\end{prop}

\begin{proof} First observe that the genus of every level of $v$ vanishes.
 The proof of (i) is by induction on the height of $v$. Assume $m=1$. Then $v_1$ is a plane. Let $\alpha$ be the asymptotic limit of $v_1$. In particular,  $\CZ(\alpha) \geq 2$, which implies
$\ind(v_1)=\CZ(\alpha)-1 \geq 1.$
Here, $\tau$ is a symplectic trivialization of the contact structure $\xi$ along $\alpha$ induced by a trivialization of $v_1^*\xi$.

Now assume $m>1$ and that (i) holds for buildings of height up to $m-1$. Let $n$ be the number of negative ends of $v_1$. The holomorphic building $(v_2,\ldots,v_m)$ is the union of $n$ genus zero holomorphic buildings, each one of them having one positive end  and no negative end. The induction hypothesis implies that
\begin{equation}
\ind(v) \geq \ind(v_1)+n.
\label{eq:n1}
\end{equation}
Since $v$ has only one positive end, $v_1$ has only one component. Let $\mathfrak{v}_1$ be the somewhere injective curve covered by $v_1$. Consider the following three cases:
\begin{itemize}
\item {\bf Case 1.} $\mathfrak{v}_1$ is a trivial cylinder.

Since the first level is not a trivial cylinder, we have $n>1$. Lemma \ref{1.7} implies that $\ind(v_1) \geq 0$. From \eqref{eq:n1} we conclude that
$\ind(v) \geq n > 1.$

\item {\bf Case 2.} $\mathfrak{v}_1$ is a non-trivial cylinder.

Lemma \ref{2.3} implies that $\ind(v_1) \geq n$. Thus \eqref{eq:n1} implies that 
$\ind(v) \geq 2n \geq 2.$

\item {\bf Case 3.} $\mathfrak{v}_1$ is not a cylinder.

In this case, $n>1$. Proposition \ref{A} implies that
$\ind(v) \geq 3-n+n=3.$
\end{itemize}
We have $\ind(v)>1$ in all three cases. Item (i) is proved.

In the proof of (i), we observe that  $m>1$ implies $\ind(v)>1$. Hence,  $\ind(v)=1$ implies that $v$ has only one level consisting of a plane. This proves (ii).

To prove (iii), assume that $\ind(v)=2$. If $m=1$, then $v$ is a plane as in (C1). Suppose now that $m>1$. As in the proof of (i), let $n$ be the number of negative ends of $v_1$, and let $\mathfrak{v}_1$ be the somewhere injective curve covered by $v_1$. Since $\ind(v)=2$, we see from the proof of (i) that $\mathfrak{v}_1$ is necessarily a cylinder. Consider the following two cases:
\begin{itemize}
 \item  $\mathfrak{v}_1$ is a trivial cylinder over $\alpha$. Since $v_1$ is not a trivial cylinder, we have $n>1$.
 Lemma \ref{1.7} and item (i) imply that
 \begin{equation*}
 2=\ind(v) \geq \ind(v_1)+n \geq n \geq 2 \,\, \Rightarrow \,\, n=2 \ \mbox{ and } \ \ind(v_1)=0.
 \end{equation*}
 Denote by $\alpha^{d_1}$ and $\alpha^{d_2}$ the asymptotic limits of $v_1$ at its negative ends. Then $v_1$ is asymptotic to $\alpha^{d_1+d_2}$ at its positive end and $(v_2, \cdots, v_m)$ is the union of two buildings, $B_1$ and $B_2$, whose asymptotic limits are $\alpha^{d_1}$ and $\alpha^{d_2}$ at their respective positive ends. Both buildings have no negative ends.
 
 From (i), we conclude that $\ind(B_i) \geq 1, i=1,2$. Since $\ind(v_1)=0$, we obtain
 $$2=\ind(v)=\ind(B_1)+\ind(B_2)\geq 2 \,\,\Rightarrow \,\, \ind(B_1)=\ind(B_2)=1.$$
 
 In view of (ii), we conclude that $B_1$ and $B_2$ are index-$1$ planes (or one of them is an index-$1$ plane and the other has two levels with a trivial cylinder in the first level and an index-$1$ plane in the second level).  This gives us two options for the building $(v_2, \ldots, v_m)$:
 \begin{enumerate}
 \item $m=2$ and $v_2$ is formed by two index-$1$ planes.
 \item $m=3$, where $v_2$ is formed by an index-$1$ plane and a trivial cylinder, and $v_3$ is an index-$1$  plane.
   
 \end{enumerate}
 In both cases, $\alpha^{d_2}$ is the asymptotic limit of an index-$1$ plane. We thus have
 $\CZ(\alpha^{d_2})=2,$
 where $\tau$ is a trivialization of the contact structure $\xi$ along $\alpha^{d_2}$ induced by a plane asymptotic to $\alpha^{d_2}$.
 This implies that $\alpha$ is hyperbolic and, therefore,
 \begin{eqnarray*}
 \ind(v_1)&=& 1+\CZ(\alpha^{d_1+d_2})-\CZ(\alpha^{d_1})-\CZ(\alpha^{d_2})\\
 &=& 1+(d_1+d_2)\cdot \CZ(\alpha)-d_1\cdot \CZ(\alpha)-d_2 \cdot\CZ(\alpha)\\
 &=&1,
 \end{eqnarray*}
 a contradiction. We conclude that if $\ind(v)=2$, then $\mathfrak{v}_1$ is not a trivial cylinder.

 \item  $\mathfrak{v}_1$ is a non-trivial cylinder. Lemma \ref{2.3} and (i) imply that
 $$2=\ind(v)\geq \ind(v_1)+n \geq 2n \geq 2 \,\,\Rightarrow \,\, n=1 \ \mbox{ and } \ind(v_1)=1. $$
 Hence $v_1$ is an index-$1$ cylinder, and the building $(v_2,\ldots,v_m)$ has index 1, a positive end, and no negative end. Therefore, (ii) implies that $m=2$ and $v_2$ is a plane. We conclude that $v$ is a building as in \ref{itm:Bc2}.
\end{itemize}
The proof of (iii) is finished.
\end{proof}

In the next theorem, we consider buildings with precisely one positive end and one negative end.

\begin{theorem} \label{C} 
Let $\lambda$ be a nondegenerate and weakly convex  contact form on $(M,\xi)$. Assume that every contractile index-$2$ Reeb orbit is embedded. Let $J\in \mathcal{J}_{\rm reg}(\lambda)$, and let $v=(v_1, \ldots, v_m)$ be a non-trivial genus zero $J$-holomorphic building with one positive end and one negative end. Then the following statements hold:
\begin{itemize}
\item[(i)] $\ind(v) \geq 1$.
\item[(ii)] If $\ind(v)=1$, then $v$ is a non-trivial cylinder.
\item[(iii)] If $\ind(v)=2$, then one of the following cases holds:
  \begin{enumerate}[label=(iii.\arabic*)]
  
   \item \label{itm:i} $v$ is a non-trivial cylinder.
   
   \item \label{itm:ii} $v=(v_1,v_2)$, where $v_1$ and $v_2$ are non-trivial cylinders of index $1$.
   
   \item \label{itm:iii} $v=(v_1,v_2)$, where 
    \begin{itemize}
   \item $v_1$  is a pair of pants with one positive end and two negative ends. It covers a trivial cylinder over a Reeb orbit $\alpha$. The index of $v_1$ is zero, its degree is $d_1 + d_2$, and the asymptotic limits at the negative ends are  $\alpha^{d_1}$ and $\alpha^{d_2}$;
   \item $v_2$ has two components:  a trivial cylinder over $\alpha^{d_1}$, and an index-$2$ holomorphic plane asymptotic to $\alpha^{d_2}$.
    \end{itemize}
    
   \item \label{itm:iv} $v=(v_1,v_2)$, where
    \begin{itemize}
     \item $v_1$  is a pair of pants with one positive end and two negative ends. It covers a trivial cylinder over a Reeb orbit $\alpha$. The index of $v_1$ is $1$, its degree is $d + 1$, and the asymptotic limits at the negative ends are  $\alpha^{d}$ and $\alpha$.
   \item $v_2$ has two components:  a trivial cylinder over $\alpha^{d}$, and an index-$1$ holomorphic plane asymptotic to $\alpha$.
    \end{itemize}

   \item \label{itm:v} $v=(v_1,v_2,v_3)$, where
    \begin{itemize}
   \item $v_1$  is a pair of pants with one positive end and two negative ends. It covers a trivial cylinder over a Reeb orbit $\alpha$. The index of $v_1$ is $0$, its degree is $d_1 + d_2$, and the asymptotic limits at the negative ends are  $\alpha^{d_1}$ and $\alpha^{d_2}$.
   \item $v_2$ has two connected components:  a trivial cylinder over $\alpha^{d_1}$, and a non-trivial index-$1$ holomorphic cylinder with a positive end at $\alpha^{d_2}$ and a negative end at a Reeb orbit $\widetilde \alpha$.
   \item $v_3$ has two connected components:  a trivial cylinder over $\alpha^{d_1}$, and an index-$1$ holomorphic plane with a positive end at  $\widetilde{\alpha}$.
    \end{itemize}
    
   \item \label{itm:vi} $v=(v_1,v_2, \ldots, v_{d_0})$ for some $2\leq d_0 \leq d+1$, where
   \begin{itemize}
   \item $v_1$ is an index-$(2-d)$ multiply covered curve with no branching points, one positive end asymptotic to a Reeb orbit $\alpha_1^d$, one negative end asymptotic to a Reeb orbit $\alpha_2^d$, and $d$ negative ends asymptotic to an index-$2$ Reeb orbit $\alpha_3$. This curve $d$-covers an index-$1$ somewhere injective pair of pants  $\mathfrak{v}_1$, with one positive end at  $\alpha_1$, and two negative ends at $\alpha_2$ and $\alpha_3$, respectively.
   
   \item $v_2$ has a trivial cylinder over $\alpha_2^d$, some index-$1$ holomorphic planes asymptotic to $\alpha_3$ and some trivial cylinders over $\alpha_3$.

   \item $v_j, i=3\ldots d_0-1,$ has some trivial cylinders over $\alpha_3$ and some index-$1$ holomorphic planes asymptotic to $\alpha_3$. The curve $v_{d_0}$ is formed by index-$1$ planes asymptotic to $\alpha_3$.
   
    \end{itemize}
    
  \end{enumerate}
\end{itemize}
\end{theorem}

\begin{figure}[ht]
\centering
{\subfigure[item \ref{itm:i}\label{fig:p1}]{\includegraphics[scale=0.17]{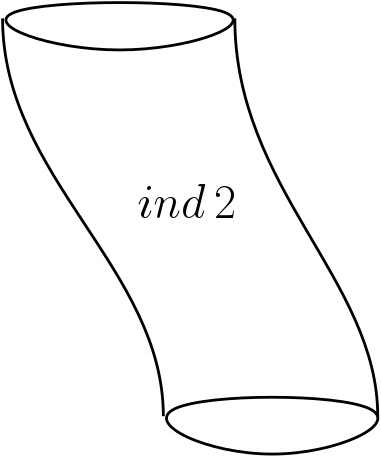}}}
\qquad 
{\subfigure[item \ref{itm:ii}\label{fig:p2}]{\includegraphics[scale=0.18]{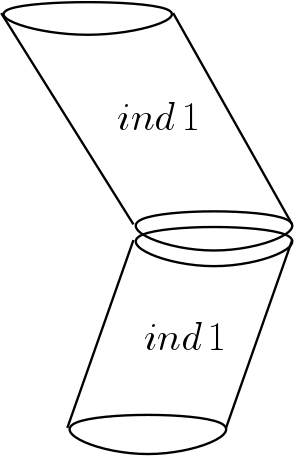}}}
\qquad 
{\subfigure[item \ref{itm:iii}\label{fig:pri}]{\includegraphics[scale=0.3]{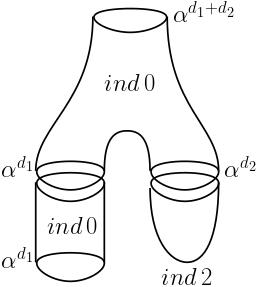}}}
\qquad 
{\subfigure[item \ref{itm:iv}\label{fig:seg}]{\includegraphics[scale=0.2]{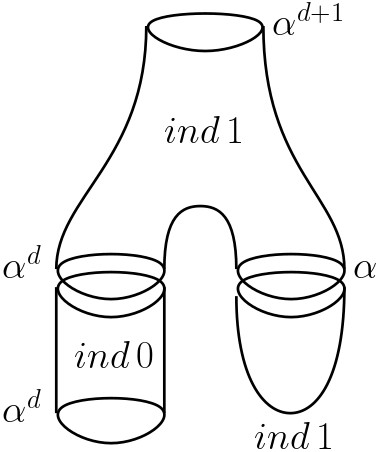}}}\\
{\subfigure[item \ref{itm:v}\label{fig:ter}]{\includegraphics[scale=0.27]{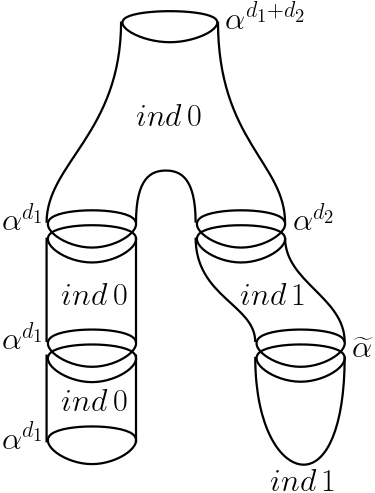}}}
\qquad
{\subfigure[item \ref{itm:vi}\label{fig:qua}]{\includegraphics[scale=0.32]{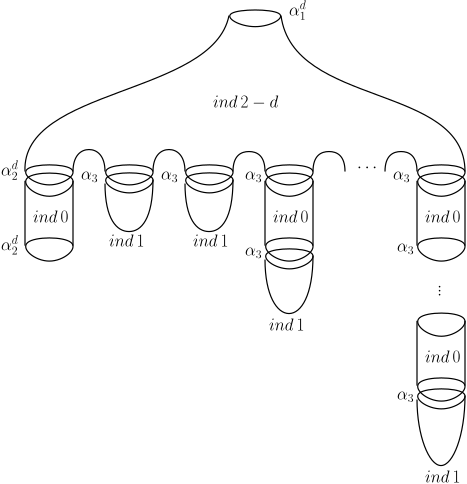}}}
\caption{$J$-holomorphic buildings in Theorem \ref{C}-(iii)}\label{fig:conj}
\end{figure}

\begin{proof}
 The proof of (i) is by induction on $m$. If $m=1$, then $v$ is a non-trivial cylinder. Lemma \ref{2.5} implies $\ind(v) \geq 1$. So assume that $m>1$ and that (i) is proved for buildings of height up to $m-1$. Let $n>0$ be the number of negative ends of $v_1$. The building $(v_2, \ldots, v_m)$ is the union of $n$ genus zero holomorphic buildings $A_1, \ldots, A_n,$ each one of them with a positive end corresponding to a negative end of $v_1$. Without loss of generality we may assume that $A_1$ has a negative end and that $A_2, \ldots, A_n$ have no negative ends.

If $n=1$, then $A_1$ is a non-trivial building and by the induction hypothesis we have $\ind(A_1)\geq 1$. The same holds for $v_1$ and, therefore,
 $\ind(v)=\ind(v_1)+\ind(A_1)\geq 2.$ 

Now suppose $n>1$. If $A_1$ is a trivial building (formed by trivial cylinders), then $\ind(A_1)=0$. If $A_1$ is a non-trivial building, then by the induction hypothesis $\ind(A_1)\geq 1$. We conclude that $\ind(A_1)\geq 0$. Theorem \ref{B} gives
\begin{equation}
\ind(v)=\ind(v_1)+\sum_{i=1}^n \ind(A_i) \geq \ind(v_1)+n-1
\label{eq:n2}
\end{equation}
Let $\mathfrak{v}_1$ be the somewhere injective curve that is covered by $v_1$. Consider the following cases:
\begin{itemize}
\item {\bf Case 1.} $\mathfrak{v}_1$ is a trivial cylinder. 

By Lemma \ref{1.7} and \eqref{eq:n2} we obtain
$\ind(v)\geq n-1 \geq 1.$

\item {\bf Case 2.} $\mathfrak{v}_1$ is a non-trivial cylinder. 

Lemma \ref{2.3} and  \eqref{eq:n2} imply
$\ind(v)\geq n+n-1=2n-1\geq 3.$

\item {\bf Case 3.} $\mathfrak{v}_1$ is not a cylinder.

By Proposition \ref{A} and  \eqref{eq:n2}, we obtain
$\ind(v)\geq 3-n+n-1=2.$
\end{itemize}
Therefore, in all three cases we have $\ind(v)\geq 1$ and (i) is proved.

To prove (ii), suppose $\ind(v)=1$. Observing  the proof of (i), we conclude that $v$ is either a non-trivial cylinder, or  $v=(v_1,\ldots,v_m)$ is a building with $m>1$, $n>1$, and $v_1$ covers a trivial cylinder over a Reeb orbit $\alpha$.

Let us show the second case above does not occur. Using the same notation as in (i), we see from  \eqref{eq:n2} and  Lemma \ref{1.7} that
$$1=\ind(v)= \ind(v_1) + \sum_{i=1}^n \ind(A_i) \geq \ind(v_1)+n-1 \geq n-1 \geq 1 \, \, \Rightarrow \,\, n=2.$$
On the other hand, since $\ind(v_1),\ind(A_1)\geq 0$ and $\ind(A_2)\geq 1$, we have
\begin{eqnarray}
1=\ind(v) =\ind(v_1)+\ind(A_1)+\ind(A_2) \geq 1, 
\end{eqnarray}
which implies 
$\ind(v_1)=\ind(A_1)=0$ and $\ind(A_2)=1$. 
In particular, $A_1$ is a trivial building and $A_2$ has only one level,  formed by a plane.

We conclude that  $v=(v_1,v_2)$, where $v_1$ covers a trivial cylinder over $\alpha$, it has a positive end at $\alpha^{d_1+d_2}$ and two negative ends at $\alpha^{d_1}$ and $\alpha^{d_2 }$. Moreover, $v_2$ is the union of the trivial cylinder over $\alpha^{d_1}$ and an index-$1$ plane  with a positive end at $\alpha^{d_2}$. Since $\ind(A_2) = 1$, we obtain $\mu_{CZ}^\tau(\alpha^{d_2})=2,$ where $\tau$ is a trivialization of the contact structure $\xi$ along $\alpha^d_2$ induced by the plane. Our assumptions on the index-$2$ Reeb orbits imply that $d_2=1$ and $\alpha$ is a simple, contractible, hyperbolic Reeb orbit. The trivialization $\tau$ induces trivializations along the asymptotic limits $\alpha^{d_1+1},\alpha^{d_1}$, also denoted by $\tau$, and a trivialization of $v_1^*\xi$. We compute

\begin{eqnarray*}
\ind(v_1) &=& 1+\CZ(\alpha^{d_1+1})-\CZ(\alpha^{d_1})-\CZ(\alpha)\\
&=& 1 + (d_1+1)\mu^{\tau}_{CZ}(\alpha)-d_1 \cdot \mu^{\tau}_{CZ}(\alpha)-  \mu^{\tau}_{CZ}(\alpha)\\
&=& 1,
\end{eqnarray*}
which is a contradiction. Therefore, if $\ind(v)=1$, then $v$ is a non-trivial cylinder. This completes the proof of (ii).

To prove (iii), suppose that $\ind(v)=2$. From the proof of (i) one of the following cases holds:

\begin{enumerate}

\item $m=1$ and $v$ is a non-trivial cylinder of index 2 as in \ref{itm:i}; or

\item $m>1, n=1$, i.e. $v_1$ and $A_1$ are non-trivial; or

\item $m>1, n>1$ and $v_1$ covers a trivial cylinder; or

\item $m>1, n>1$ and $v_1$ covers a curve that is not a cylinder.
\end{enumerate}

Let's take a closer look at the last three cases.
\begin{itemize}
\item If Case 2 holds, then (i) implies that $\ind(v_1)\geq 1$ and $\ind(A_1)\geq 1$. Then,
$$2=\ind(v)=\ind(v_1)+\ind(A_1) \,\, \Rightarrow \,\, \ind(v_1)=\ind(A_1)=1.$$
It follows from (ii) that $v_1$ and $A_1$ are non-trivial cylinders, as in \ref{itm:ii}.

\item If Case 3 holds, then we  consider two sub-cases: 
\begin{itemize}
 \item $A_1$ is trivial. In this case,  $\ind(A_1)=0$. By Theorem \ref{B} and Lemma \ref{1.7} we obtain
 $$2=\ind(v)=\ind(v_1)+\sum_{i=2}^n \ind(A_i) \geq n-1.$$
 We conclude that
 $2 \leq n \leq 3 $, i.e. $n=2$ or  $n=3$.
 If $n=2$, then
 \begin{equation}
 2=\ind(v)=\ind(v_1)+\ind(A_2),
 \label{eq:n3}
 \end{equation}
 where $\ind(v_1)\geq 0$ and $\ind(A_2) \geq 1.$
We have the following possibilities:
 \begin{itemize}
 \item[(a)] $\ind(v_1)=0$ and $\ind(A_2)=2$; or
 \item[(b)] $\ind(v_1)=1$ and $\ind(A_2)=1$.
 \end{itemize}
 
 If (a) holds, then by Theorem \ref{B}, either $A_2$ is a plane and $v$ is as in \ref{itm:iii}, or $A_2$ is a building with two levels and $v$ is as in  \ref{itm:v}.
 
 If (b) holds, then  Theorem \ref{B} implies that $A_2$ is a plane and $v$ is as in \ref{itm:iv}.
 
 If $n=3$, then
 $2=\ind(v_1)+\ind(A_2)+\ind(A_3),$
 where $\ind(v_1)\geq 0$ and $\ind(A_2),\ind(A_3)\geq 1$. It follows that
 $\ind(v_1)=0$  and $\ind(A_2)=\ind(A_3)=1.$
 From Theorem \ref{B}-(ii
88) we conclude that $A_2$ and $A_3$ are planes. Since $v_1$ has three negative ends, say in $\alpha^{d_1},\alpha^{d_2}$ and $\alpha^{d_3}$, we can assume without loss of generality that $A_2$ and $ A_3$ are asymptotic to $\alpha^{d_2}$ and $\alpha^{d_3}$ at their positive ends, respectively. Our hypotheses on the index-$2$ orbits imply that $\alpha$ is simple and $d_2=d_3=1$.  In particular, $\alpha$ is hyperbolic and $\CZ(\alpha)=2$. Taking a trivialization $\tau$ of $\xi$ along $\alpha$ we obtain an induced trivialization of $v_1^*\xi$. This trivialization, also denoted by $\tau$, induces a trivialization of $\xi$ along each of its asymptotic limits, which are iterated from $\alpha$.  So,
 \begin{eqnarray*}
 \ind(v_1) &=& -(2-4)+\CZ(\alpha^{d_1+d_2+d_3})-\CZ(\alpha^{d_1})-\CZ(\alpha^{d_2} )-\CZ(\alpha^{d_3})\\
 &=& 2 + (d_1+2)\mu^{\tau}_{CZ}(\alpha)-d_1 \mu^{\tau}_{CZ}(\alpha)- \mu^{\tau}_{CZ}(\alpha)-  \mu^{\tau}_{CZ}(\alpha)\\
 &=& 2,
 \end{eqnarray*}
  contradicting  $\ind(v_1)=0$. We conclude that $v_1$ cannot have three negative ends and the case $n=3$ is ruled out.

 \item $A_1$ is non-trivial. In this case, item (i)  implies that $\ind(A_1)\geq 1$. By Theorem \ref{B} and  Lemma \ref{1.7} we obtain
 $$2=\ind(v)=\ind(v_1)+\ind(A_1)+\sum_{i=2}^n \ind(A_i) \geq 1+n-1=n \geq 2,$$
 which implies  $n=2$. On the other hand, since $\ind(v_1)\geq 0$ and $\ind(A_1),\ind(A_2)\geq 1$, we obtain
 $$2=\ind(v_1)+\ind(A_1)+\ind(A_2) \,\, \Rightarrow \,\, \ind(v_1)=0 \ \ \mbox{ and } \ \ \ind(A_1)=\ind (A_2)=1.$$
 In particular, $A_2$ is an index-$1$ plane asymptotic to a simple index-$2$ hyperbolic Reeb orbit. As in the proof of  (ii), since $v_1$ covers a trivial cylinder and $n=2$, we conclude that $\ind(v_1) =1$, a contradiction. Therefore, this case is ruled out.
 \end{itemize}

 \item If Case 4 holds, then Theorem \ref{B} and Proposition \ref{A} imply
$$2=\ind(v)=\ind(v_1)+\ind(A_1)+\sum_{i=2}^n \ind(A_i) \geq 3-n+\ind(A_1)+n-1= \ind(A_1)+2,$$
and thus  $\ind(A_1)=0$. Item (i) implies that $A_1$ is a trivial building, i.e., $A_1$ is formed by trivial cylinders. It also follows from the inequality above that $\ind(v_1)=3-n$. 
By Proposition \ref{A},  $v_1$ has index $2-d$, it covers (with multiplicity $d=n-1$ and no branching points) an index-$1$ somewhere injective curve $\mathfrak{v}_1$, with one positive end at $\alpha_1$ and two negative ends, one at $\alpha_2$ and the other at $\alpha_3$. Moreover, $v_1$ has $n=d+1$ negative ends.

We compute
\begin{equation}
2=\ind(v)=\ind(v_1)+\sum_{i=2}^{d+1}\ind(A_i)=2-d+\sum_{i=2}^{d+1}\ind(A_i) \,\, \Rightarrow \,\, \sum_{i=2}^{d+1}\ind(A_i)=d.
\label{eq:n4}
\end{equation}

Theorem \ref{B} gives $\ind(A_i)\geq 1 \forall i=2,\ldots,d+1,$ and hence \eqref{eq:n4} implies  $\ind(A_i)=1 \forall i=2,\ldots, d+1$. In particular, $A_1$ is a trivial cylinder, and $A_i,i=2,\ldots, d+1$ consists of a pile of trivial cylinders over the same orbit and, at the lowest level, an index-$1$ plane asymptotic to the same orbit. 

We can assume, without loss of generality, that $A_1$ is a trivial cylinder over $\alpha_2^p$, $A_i,i=2,\ldots,d+1,$ are planes, and $A_{d+1}$ is asymptotic to  $\alpha_3^q$ for some $1\leq p,q\leq d$.

Denote by $\beta_i$ the asymptotic limit of the positive end of the planes $A_i$, $i=2,\ldots, d+1$. It follows that such orbits are contractile. Fix a trivialization $\tau_i$ of $\xi$ over $\beta_i$ that extends to a trivialization of $\xi$ over $A_i$. Since $\ind(A_i)=1$, it follows that $\mu_{CZ}^{\tau_i}(\beta_i)=2$. Now we use the hypothesis that every contractile Reeb orbit with index 2 is simple
to conclude that the periodic orbits $\beta_i,i=2,\ldots,d+1,$ are contractile and simple. In particular, since the plane $A_{d+1}$ is asymptotic to $\alpha_3^q$, we conclude that $\beta_{d+1}=\alpha_3$ and $q=1$.

Since $v_1$ is a $d$-cover of $\mathfrak{v}_1$, and each $\beta_2,\ldots,\beta_{d+1}$ are simple, $v_1$ has at least $d$ negative ends asymptotic to $\alpha_3$. This forces $\beta_i=\alpha_3, \,\, \forall i=2,\ldots,d+1,$ 
and $p=d$.
We conclude that the asymptotic limit of $A_i,i=2,\ldots,d+1,$ is a common index-$2$ simple orbit $\alpha_3$, and $A_1$ is a trivial cylinder over $\alpha_2^d$.

In short, $v$ is a building such that:
\begin{itemize}
     
     \item $v_1$ is a curve with one positive end at $\alpha_1^d$, one negative end at  $\alpha_2^d,$ and $d$ negative ends at a simple contractible index-$2$ orbit $\alpha_3$.
     
     \item $v_2$ is the union of a trivial cylinder over $\alpha_2^d$, some trivial cylinders over $\alpha_3$ and some  planes asymptotic to $\alpha_3$.
     
     \item There exists $2\leq d_0\leq d+1$ such that $v_j,j=3,\ldots,d_0,$ is formed by cylinders over $\alpha_3$ and planes asymptotic to $\alpha_3$. The lowest level $v_{d_0}$ contains only planes asymptotic to $\alpha_3$.  
\end{itemize}
This building is as in \ref{itm:vi}.
\end{itemize}
The proof is finished. \end{proof}

\subsection{Ruling out buildings as in Theorem \ref{C}-(iii.4) and (iii.5)}\label{sec_ruling_out_buildings}

In this section, we  rule out certain index-$2$ buildings that might appear at the end of an index-$2$ family of holomorphic cylinders. The argument follows the same lines in the proof of the following theorem by Hutchings and Nelson.

\begin{prop}[Hutchings-Nelson \cite{HN}, Proposition 3.1] \label{3.1} 
Let $\lambda$ be a nondegenerate contact form on a closed $3$-manifold $M$, and let $J \in \mathcal{J}(\lambda)$. Let $u=(u_1,u_2)$ be a $J$-holomorphic building such that:
\begin{itemize}
    \item[(i)] $u_1$ is an index-$0$ curve that multiply covers a trivial cylinder over a Reeb orbit $\alpha$. The covering number is $d+1$, and $u_1$  has a positive end at $\alpha^{d+1}$, and two negative ends at $\alpha^{d}$ and $\alpha$.
    
    \item[(ii)] $u_2$ has two components, one of them is a trivial cylinder over $ \alpha^{d}$, and the other is an index-$2$ plane with a positive end at $\alpha$.
\end{itemize}
Then $u$ is not in the closure of the moduli space $\M_{0,2}^J(\alpha^{d+1},\alpha^d)/\R$.
\end{prop}

We shall rule out buildings of the following types:
\begin{itemize}
 \item Case 1. $u=(u_1,u_2)$ is a $J$-holomorphic building where:
\begin{itemize}
\item $u_1$ is an index-$1$ curve that multiply covers a trivial cylinder over $\alpha$. The covering number is $d+1$, and $u_1$ has  a positive end at $\alpha^{d+1}$, and two negative ends at $\alpha^d$ and $\alpha$.
\item $u_2$ has two components, one of them is a trivial cylinder over $\alpha^d$, and the other is an index-$1$ plane  with positive end at $\alpha$.
\end{itemize}
\item Case 2. $u=(u_1,u_2, u_3)$ is a $J$-holomorphic building where:
 \begin{itemize}
   \item $u_1$ is an index-$0$ curve that multiply covers a trivial cylinder over $\alpha$. The covering number is $d+1$, and $u_1$ has a positive end at $\alpha^{d+1}$, and two negative ends at $\alpha^d$ and $\alpha$. 
   \item $u_2$ has two  components, one of them is a trivial cylinder over $\alpha^{d}$, and the other is an index-$1$ cylinder with a positive end at $\alpha$ and a negative end at $\tilde{\alpha}$.
   \item $u_3$ has two  components: a trivial cylinder over $\alpha^{d}$, and an index-$1$ plane  with a positive end at $\tilde{\alpha}$.
    \end{itemize}
\end{itemize}

\begin{figure}[ht]
\centering
{\subfigure[Case 1]{\includegraphics[scale=0.2]{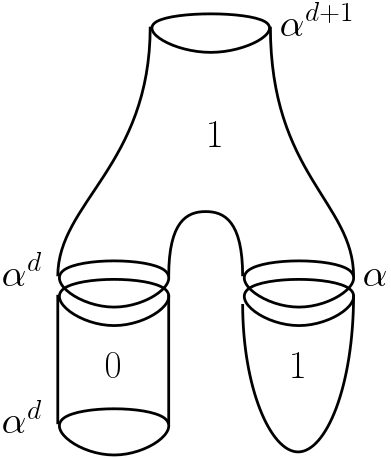}}}
\hspace{2cm}
{\subfigure[Case 2]{\includegraphics[scale=0.22]{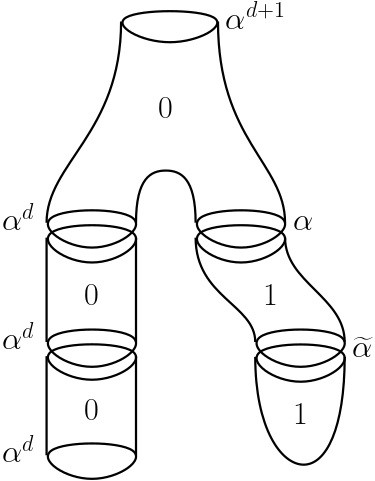}}}
\caption{Holomorphic buildings of Cases 1 and 2.}
\end{figure}

Throughout this subsection, we assume that the Reeb orbits $\alpha$ and $\tilde \alpha$, as in Cases 1 and 2 above, are embedded. We introduce some notations before studying sequences of holomorphic cylinders converging to buildings like those in Cases 1 and 2 above.

\begin{definition}[Hutchings {\cite[Definition 2.7]{H2009}}]
A braid around an embedded Reeb orbit $\alpha$ is an oriented link $\zeta$ contained in a tubular neighborhood $N$ of $\alpha$ such that the projection $\zeta \to \alpha$ is a submersion that preserves the orientation.
\end{definition}

\begin{definition}
Let $\gamma=\pi \circ \tilde \gamma:S^1 \to \R^2$ be the projection of an oriented link $\tilde \gamma:S^1 \to \R^3$, where $ \pi:\R^3 \to \R^2$ is a projection. We assume that $\gamma$ contains only generic double self-intersections. Each self-intersection of $\gamma$ is assigned a value $+1$ or $-1$ according to its type. We use the convention that counterclockwise twists are positive. The {\bf writhe} of the oriented link $\tilde \gamma$ is defined as the sum of all such values and is denoted $w(\gamma)$. 
\end{definition}


\begin{definition}[Hutchings {\cite[Definition 2.8]{H2009}}] 
Let $\alpha$ be an embedded Reeb orbit, let $\tau$ be a trivialization of $\alpha^*\xi$, and let $\zeta$ be a braid around $\alpha$. Extend the trivialization $\tau$ to obtain a diffeomorphism $\phi_{\tau}: V \to S^1 \times D$ between a tubular neighborhood $V$ of $\alpha$ and $S^1 \times D$ so that the projection $\zeta \to S^1$ is a submersion. We also identify $S^1 \times D$ with a solid torus of $\R^3$ through the orientation-preserving diffeomorphism $f: S^1 \times D \to \R^3$, given by
$\left(\theta, x,y\right) \mapsto \left(1+x/2\right)(\cos \theta, \sin \theta,0)-\left(0,0,y/2\right).$
Let $p:\R^3 \to \R^2\times \{0\}$ be the canonical projection onto the plane $\theta x$. We define the writhe of $\zeta$ and denote it by $w_{\tau}(\zeta) \in \Z$ as the writhe of the oriented link $p \circ f \circ \phi_{\tau}(\zeta) $ in $\R^2$, i.e.,
$w_{\tau}(\zeta):=w(p \circ f \circ \phi_{\tau}(\zeta)).$
\end{definition}

\begin{definition}
Let $u=(a,v):\Sigma \setminus \Gamma \to \R \times M$ be a $J$-holomorphic curve, where $J \in \mathcal{J}(\lambda)$. Suppose that $u$ has a positive end at an embedded Reeb orbit $\alpha$. In a tubular neighborhood of $\alpha$ in $M$, adapted to a trivialization $\tau$ of $\alpha^*\xi$, we write  $v(s,t)=(\theta(s,t),z(s,t))$, where $\theta(s,t) \in S^1$ and $z(s,t) \in \R^2$. Let $t\mapsto \zeta_{R_+}(t):=(a(R_{+},t),\theta(R_{+},t),z(R_{+},t)), \ \ t\in S^1,$ be the braid obtained by intersecting $u\left([R,\infty) \times S^1\right)$ with $\{R_{+}\}\times M $, for $R_{+}>0$ large enough. We define
\begin{equation}
\wind(\zeta):={\text wind}(t \mapsto z(R_{+},t)).
\label{eq:defwind}
\end{equation}
Note that $\wind(\zeta)$ is independent of $R_{+}$, if $R_{+}$ is large enough.

Similarly, if $u$ has a negative end at an embedded Reeb orbit  $\beta$ and $\zeta$ is the braid obtained by intersecting $u\left((-\infty,0] \times S^1\right)$ with $\{ R_{-}\}\times M$, for sufficiently large $-R_{-}$, we define
$\wind(\zeta):={\rm wind}(t \mapsto z(R_{-},t)).$
\end{definition}

\begin{lemma}[Hofer-Wysocki-Zehnder \cite{prop2}] \label{wind pi} 
Let $u:\C \to \R \times M$ be a $J$-holomorphic plane asymptotic to a nondegenerate Reeb orbit $\alpha$. Consider a trivialization $\tau$ of $\alpha^*\xi$ that extends to a trivialization of $u^*\xi$. Let $\zeta$ be the intersection of $u(\C)$ with $\{R_{+}\}\times M$. If $R_{+} \gg 0$, then
$\wind(\zeta) \geq 1.$
\end{lemma}

\begin{lemma}[Hutchings-Nelson {\cite[Lemma 3.2]{HN}}] \label{L3.2} 
Let $u:\Sigma \setminus \Gamma \to \R \times M$ be a $J$-holomorphic curve, where $J \in \mathcal{J}(\lambda)$. Assume that $u$ is not part of a trivial cylinder or a multiply-covered component. Let $\alpha$ be an embedded Reeb orbit and assume that $u$ has a positive end at $\alpha^d$, for some $d\geq 1$. Let $\zeta$ be the intersection of this positive end of $u$ with $\{R_{+}\}\times M$. If $R_{+} \gg 0$, then
$$w_{\tau}(\zeta) \leq (d-1)\wind(\zeta) \leq (d-1)\lfloor\CZ(\alpha^d)/2\rfloor.$$ 
\end{lemma}

\begin{lemma}[Hutchings-Nelson {\cite[Lemma 3.3]{HN}}] \label{L3.4} 
Let $u:\Sigma \setminus \Gamma \to \R \times M$ be a $J$-holomorphic curve, where $J \in \mathcal{J}(\lambda)$. Assume that $u$ is not part of a trivial cylinder or a multiply-covered component.
Let $\alpha$ be an embedded Reeb orbit and assume that $u$ has a negative end at $\alpha^d$, $d\geq 1$. Let $\zeta$ be the intersection of this negative end of $u$ with $\{R_{-}\}\times M$. If $R_{-} \ll 0$, then
$$w_{\tau}(\zeta) \geq (d-1)\wind(\zeta) \geq (d-1)\lceil\CZ(\alpha^d)/2\rceil.$$ 
\end{lemma}

\begin{lemma}[Hutchings-Nelson {\cite[Lemma 3.5]{HN}}] \label{L3.5} 
Let $\alpha$ be an embedded Reeb orbit with tubular neighborhood $N$ and $\tau$ a trivialization of $\alpha^*\xi$. Let $u$ be a $J$-holomorphic curve, $J \in \mathcal{J}(\lambda)$, whose image is contained in $[R_{-},R_{+}]\times N$. Assume that $u$ has no multiply-covered components. Let $\zeta_{+}-\zeta_{-}$ be the boundary of $u$, where $\zeta_{\pm}$ is a braid in $\{R_{\pm}\}\times N$. Then,
$$\chi(u)+w_{\tau}(\zeta_{+})-w_{\tau}(\zeta_{-}) \geq 0,$$
where $\chi(u)=\chi(\dot \Sigma)$ is the Euler characteristic of the domain of $u$.
\end{lemma}

The following proposition is an adaptation of Proposition \ref{3.1} to Case 1.

\begin{prop} \label{naopode1} 
Let $\lambda$ be a nondegenerate contact form on a closed $3$-manifold $M$,  and let $J \in \mathcal{J}(\lambda)$. Let $u=(u_1,u_2)$ be a $J$-holomorphic building as in Case 1. Then $u$ is not in the closure of $\M_{0,2}^J(\alpha^{d+1 };\alpha^d)/\R$.
\end{prop}

\begin{proof}
Let $\alpha$ be the asymptotic limit of the index-$1$ plane $v_2$ in $u_2$. Then $\CZ(\alpha)=2$, where $\tau$ is a trivialization of $\alpha^*\xi$ that extends to a trivialization of $\xi$ over $v_2^*\xi$. In particular, $\alpha$ is embedded and hyperbolic by assumption.
After translating $v_2$ downward by a sufficiently large real number if necessary, we may fix a tubular neighborhood $N\subset M$ of $\alpha$ and a real number $\varepsilon>0$ such that
${\rm dist}\left(v_2\left(v_2^{-1}(\{0\}\times N)\right),\alpha\right)\geq \varepsilon$ and for every $s\geq 0$, $v_2(v_2^{-1}(\{s\} \times N))$ is a braid around $\alpha$.

Suppose by contradiction that the sequence of $J$-holomorphic cylinders $\bar{w}_k$ in the moduli space $\M_{0,2}^J(\alpha^{d+1};\alpha^d)/\R$ converges to $u= (u_1,u_2)$. Denote by ${w}_k$ a parametrization of $\bar{w}_k$. 
For $k$ sufficiently large, the following holds:

\begin{itemize}
    \item ${w}_k^{-1}\left((-\infty,0] \times M\right)$ has two components:  a half-cylinder $C_k$ and a closed disk $D_k$;
    
    \item $\zeta_1^k=w_k(C_k) \cap (\{0\}\times N)$ is a braid whose projection to $\alpha$ has degree $d$ and ${\rm dist}(\zeta_1^k,\alpha) \leq \varepsilon/3;$
    
    \item $\zeta_2^k=w_k(D_k)\cap (\{0\}\times N)$ is a braid whose projection to $\alpha$ has degree $1$ and
  $${\rm dist} \left(\zeta_2^k,v_2\left(v_2^{-1}(\{0\}\times N)\right)\right) \leq \varepsilon/3,$$
    \item $\zeta_{+}^k={w}_k\big(w_k^{-1}([0,+\infty)\times M)\big) \cap (\{R_{+}\}\times N)$, $R_{+} \gg 0$ depending on $k$, is a braid whose projection to $\alpha$ has degree $d+1$,
    \item $\zeta_{-}^k={w}_k\big(w_k^{-1}((-\infty,0]\times M)\big) \cap (\{R_{-}\}\times N)$, $R_{-}\ll 0$ depending on $k$, is a braid whose projection to $\alpha$ has degree $d$.
\end{itemize}

\begin{figure}[ht]
\centering
\includegraphics[scale=0.3]{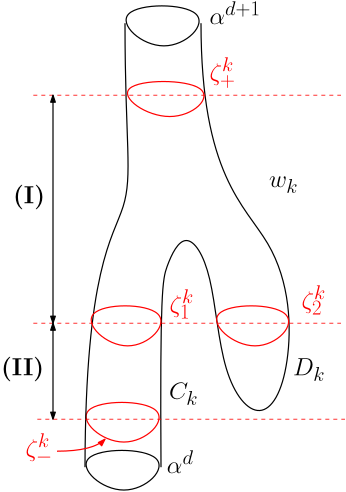}
\caption{Braids $\zeta^k_{1,2}$ and $\zeta^k_{\pm}$ in the image of the curves $w_k$ that converge to the building $u$.}
\label{fig:u_k}
\end{figure}

From the conditions above, we have ${\rm dist}(\zeta_1^k,\zeta_2^k) \geq \varepsilon/3$, and thus $\zeta_1^k$ and $\zeta_2^k$ do not intersect. Hence $\zeta_1^k \cup \zeta_2^k$ is a braid whose writhe is given by 
$
w_{\tau}(\zeta_1^k \cup \zeta_2^k)=w_{\tau}(\zeta_1^k)+2d\wind(\zeta_2^k)+w_{\tau}(\zeta_2^k).
$
Since $\zeta_2^k$ projects into $\alpha$ with degree $1$, it follows that $\zeta_2^k$ is a trivial knot, which implies that $w_{\tau}(\zeta_2^k)=0.$
Using Lemma \ref{L3.5} on the domain of $w_k$ bounded by $\zeta_{+}^k$ and $(\zeta_{1}^k \cup \zeta_2^k)$, see Figure \ref{fig:u_k}, we obtain
$$ 
-1+w_{\tau}(\zeta_{+}^k)-w_{\tau}(\zeta_1^k \cup \zeta_2^k) \geq 0.
$$
Using  Lemma \ref{L3.5} again on the domain of $w_k$ bounded by $\zeta_1^k$ and $\zeta_{-}^k$, we obtain
\begin{equation*}
w_{\tau}(\zeta_1^k)-w_{\tau}(\zeta_{-}^k) \geq 0.
\end{equation*}
We conclude that 
\begin{eqnarray*}
2d \wind(\zeta_2^k) =w_{\tau}(\zeta_1^k \cup \zeta_2^k)-w_{\tau}(\zeta_1^k) \leq -1 + w_{\tau}(\zeta_{+}^k)-w_{\tau}(\zeta_{-}^k).
\end{eqnarray*}
Now recall that $\CZ(\alpha)=2$. Using Lemma \ref{L3.2} to $w_k$, we obtain
\begin{equation*}
w_{\tau}(\zeta_{+}^k) \leq d \left\lfloor \frac{\CZ(\alpha^{d+1})}{2}\right\rfloor = d \left\lfloor (d+1)\frac{\CZ(\alpha)}{2}\right\rfloor = d(d+1).
\end{equation*}
Using  Lemma \ref{L3.2} to $v_2$, and the fact that $w_k$ converges to the building $u=(u_1,u_2)$, we obtain
$$\wind(\zeta_2^k) \leq \left\lfloor \frac{\CZ(\alpha)}{2}\right\rfloor = 1.$$
On the other hand, by the Lemma \ref{wind pi}, we have $\wind(\zeta_2^k) \geq 1$. We concluded that $\wind(\zeta_2^k)=1.$
Using Lemma \ref{L3.4} to $w_k$, we obtain
\begin{equation*}
w_{\tau}(\zeta_{-}^k) \geq (d-1)\left\lceil \frac{\CZ(\alpha^d)}{2} \right\rceil \geq (d-1)\left\lceil d \cdot\frac{\CZ(\alpha)}{2} \right\rceil =d(d-1).
\end{equation*}
Putting the inequalities above together, we finally obtain
\begin{eqnarray*}
0 &\leq & -1 + w_{\tau}(\zeta_{+}^k)-w_{\tau}(\zeta_{-}^k)-2d\wind(\zeta_2^k) \\
&\leq & -1 + d(d+1) -d(d-1)-2d\\
&=& -1,
\end{eqnarray*}
which is absurd. We conclude that there exists no sequence in $\M^J_{0,2}(\alpha^{d+1};\alpha^d)/\R$ converging to the building $u$, i.e. $u$ is not in the closure of $\M^J_{0,2}(\alpha^{d+1};\alpha^d)/\R$.
\end{proof}

Finally, we adapt Proposition \ref{3.1} to Case 2.

\begin{prop} \label{naopode2} 
Let $\lambda$ be a nondegenerate contact form on a closed $3$-manifold $M$,  and let $J \in \mathcal{J}(\lambda)$. Let $u=(u_1,u_2,u_3)$ be a $J$-holomorphic building as in Case $2$. Then $u$ is not in the closure of $\M_{0,2}^J(\alpha^{d+1};\alpha^d)/\R$.
\end{prop}

\begin{proof}
Since the $J$-holomorphic plane in $u_3$ has index $1$, we have $\CZ(\widetilde{\alpha})=2$. Since the $J$-holomorphic cylinder $v_2$ in $u_2$ has index $1$, we conclude that $\alpha$ is contractible and $\CZ(\alpha)=3$. Our hypotheses imply that both $\alpha$ and $\widetilde \alpha$ are embedded. Moreover, $$0\leq \text{wind}_\pi(v_2)=\wind_\infty(v_2;+\infty)-\wind_\infty(v_2;-\infty),$$ with $$\wind_\infty(v_2;+\infty)\leq \lfloor\cz(\alpha)/2\rfloor=1 \quad \mbox{ and } \quad \wind_\infty(v_2;-\infty)\geq  \lceil \cz(\widetilde \alpha)/2\rceil = 1.$$ We conclude that $\wind_\infty(v_2;+\infty)=\wind_\infty(v_2;-\infty)=1$. In particular, $v_2$ behaves near $\alpha$ as an index-$2$ plane asymptotic to $\alpha$. 

Now we consider sequences $\bar{w}_k \in \M_{0,2}^J(\alpha^{d+1};\alpha^d)/\R$ converging to the building $u= (u_1,u_2,u_3)$, and braids $\zeta_{1,2}^k, \zeta^k_\pm$, with $\zeta^k_{1,2}$ close to $\alpha$, to obtain a contradiction as in the proof of Proposition \ref{3.1}, see \cite[Proposition 3.1]{HN}.
\end{proof}

\subsection{Buildings as in Theorem \ref{C}-(iii.6)}

In this section we consider buildings $v=(v_1,v_2)$ as in Theorem \ref{C}-\ref{itm:vi}. Recall that $v_1$ is an index-$(2-d)$ curve with a positive end at a Reeb orbit $\alpha_1^d$, a negative end at $\alpha_2^d$, and $d$ negative ends at a common index-$2$ simple and contractible Reeb orbit $\alpha_3$. The curve $v_1$ covers a somewhere injective curve $\mathfrak{v}_1$. The covering has multiplicity $d$ and no branching points. The curve $\mathfrak{v}_1$ has a positive end at $\alpha_1$, a negative end at $\alpha_2$, and a negative end at $\alpha_3$. The second level $v_2$ is formed by a trivial cylinder over $\alpha_2^d$ and $d$ index-$1$ planes asymptotic to $\alpha_3$. 

Consider a trivialization $\tau$ of $\xi$ along $\alpha_1,\alpha_2$ and $\alpha_3$, which extends to a trivialization of $\mathfrak{v}_1^*\xi$ and $ w^*\xi$, where $w:\mathbb{D} \to M$ is a continuous disk whose boundary covers $\alpha_3$. The $d-$covering $v_1 \to \mathfrak{v}_1$ and $\tau$ induce a trivialization of $v_1^*\xi$, which extends to trivializations of $\xi $ along $\alpha_1^d$ and $\alpha_2^d$, also denoted by $\tau$. 

\begin{lemma}\label{odepeq}
Assume that $\alpha_1^d, \alpha_2^d$ are good orbits satisfying $\CZ(\alpha_1^d)-\CZ(\alpha_2^d)=2$. Then the following assertions hold:
\begin{itemize}
    \item[(i)] If $\alpha_1$ and $\alpha_2$ are both hyperbolic, then $d=1$.
    \item[(ii)] If $\alpha_1$ is elliptic and $\alpha_2$ is hyperbolic, then $\alpha_2$ is negative hyperbolic and $d=1$.
    \item[(iii)] If $\alpha_1$ is hyperbolic and $\alpha_2$ is elliptic, then $\alpha_1$ is negative hyperbolic and $d=1$.
\end{itemize}
In particular, if $d>1$, then $\alpha_1$ and $\alpha_2$ are elliptic.
\end{lemma}
\begin{proof}
We have
\begin{eqnarray*}
1&=&\ind(\mathfrak{v}_1)=-(2-3)+\CZ(\alpha_1)-\CZ(\alpha_2)-2\\
&=& \CZ(\alpha_1)-\CZ(\alpha_2)-1,
\end{eqnarray*}
which implies
$\CZ(\alpha_1)-\CZ(\alpha_2)=2.$
In particular, $\CZ(\alpha_1)$ and $\CZ(\alpha_2)$ have the same parity. Therefore,
\begin{itemize}
    \item If $\alpha_1$ and $\alpha_2$ are hyperbolic, then our hypotheses imply
     $$2=\CZ(\alpha_1^d)-\CZ(\alpha_2^d)=d(\CZ(\alpha_1)-\CZ(\alpha_2))=2d \,\,\Rightarrow \,\, d =1.$$
     \item If $\alpha_1$ is elliptic and $\alpha_2$ is hyperbolic, then we obtain
     \begin{eqnarray*}
     2&=& \CZ(\alpha_1^d)-\CZ(\alpha_2^d)=2\lfloor d\theta_{\alpha_1}\rfloor+1 -d\cdot\CZ(\alpha_2)\\
     &\geq & 2d\lfloor \theta_{\alpha_1}\rfloor+1 -d \cdot \CZ(\alpha_2)=d(2\lfloor \theta_{\alpha_1}\rfloor+1)-d+1 -d \cdot\CZ(\alpha_2)\\
     &=&d(\CZ(\alpha_1)-\CZ(\alpha_2))-d+1\\
     &=& d+1.
     \end{eqnarray*}
     Hence  $d=1$.
     \item Assume now that $\alpha_1$ is hyperbolic and $\alpha_2$ is elliptic. Let $\{x\}:=x-\lfloor x \rfloor, \forall x\in \R$. Then $\{\theta_{\alpha_2}\} \in [0,1] \backslash \Q$, which implies $0<d\{\theta_{\alpha_2}\}<d$. We estimate
    \begin{eqnarray*}
    \lceil d\theta_{\alpha_2} \rceil &=& \big\lceil d \lfloor \theta_{\alpha_2} \rfloor + d\{\theta_{\alpha_2}\} \big\rceil = d \lfloor \theta_{\alpha_2} \rfloor + \lceil d\{\theta_{\alpha_2}\} \rceil \\
    &\leq & d \lfloor \theta_{\alpha_2} \rfloor +d = d(\lfloor \theta_{\alpha_2} \rfloor + 1) \\
    &=& d\lceil \theta_{\alpha_2}\rceil,
    \end{eqnarray*}
    and
    \begin{eqnarray*}
    2&=& \CZ(\alpha_1^d)-\CZ(\alpha_2^d)= d \cdot \CZ(\alpha_1)-(2\lceil d\theta_{\alpha_2}\rceil-1)\\
    &\geq& d \cdot \CZ(\alpha_1)-2d \lceil \theta_{\alpha_2} \rceil + 1 = d \cdot \CZ(\alpha_1)-d(2 \lceil \theta_{\alpha_2} \rceil -1)-d+1 \\
    &=& d(\CZ(\alpha_1)-\CZ(\alpha_2))-d+1. \\
    &=& d+1
    \end{eqnarray*}
    We conclude as before that $d=1$.
\end{itemize}
\end{proof}

The following Lemma provides conditions for automatic transversality. 

\begin{lemma}[{\cite[Lemma 4.1]{HN}}] \label{lem4.1HN}
Let $\lambda$ be a nondegenerate contact form on a closed  $3$-manifold,  and let $J \in \mathcal{J}(\lambda)$. Let $\Sigma$ be a genus $g$ connected closed surface, $\Gamma \subset \Sigma$ a finite set, and $u:\Sigma \setminus \Gamma \to \R \times M$ an immersed finite energy $J$-holomorphic curve. Denote by $\Gamma_0(u)\subset \Gamma$ the set of punctures whose asymptotic limit has an even Conley-Zehnder index. If 
$
    \ind(u)>2g-2+\sharp\Gamma_0(u),
$
then $D_u:\Gamma(N) \to \Gamma(F)$ is surjective. See section \ref{sec_linear} for definitions.
\end{lemma}

 Lemma \ref{lem4.1HN} implies that the space of index-$1$ planes asymptotic to a given Reeb orbit is a manifold. 

\begin{lemma} \label{dim0} 
Let $\lambda$ be a nondegenerate and weakly convex contact form on a closed $3$-manifold $M$, and let $J \in \mathcal{J}(\lambda)$. Let $\alpha$ be an index-$2$ Reeb orbit. Then $\M_{0,1}^J(\alpha; \emptyset)/\R$ is a compact zero-dimensional manifold. In particular, $\M_{0,1}^J(\alpha;\emptyset)$ is a finite set.
\end{lemma}

\begin{proof}
Let $u:=(a,\bar u):\C \to \R \times M$ be an element of $\M_{0,1}^J(\alpha; \emptyset)$. To show that $u$ is an immersion, denote by $\pi:TM \to \xi:=\ker \lambda$ the projection along the Reeb vector field $R_{\lambda}$. Recall that ${\rm wind}_\pi(\bar u)$ is the algebraic sum of the local degrees of $\pi \circ d\bar u$ at its zeros, that is
$
\windd_{\pi}(\bar{u}):=\sum_{\pi \circ d\bar{u}(z)=0}{\rm deg}(X,z),
$
where $X:\C \to \C$ is a representation of $\pi \circ d\bar{u} \cdot \partial_x$ via a trivialization of $u^*\xi$.
It is proved in \cite{prop2} that 
\begin{equation}\label{windpi}
0\leq {\rm wind}_{\pi}(\bar{u})=\wind_{\infty}(\bar{u},\infty)-1,
\end{equation}
where $\wind_\infty(\bar u,\infty)$ is the winding number of the leading eigenvector of the asymptotic operator $A_{\alpha}$ at $\infty$, with respect to a trivialization $\tau$ of $\xi$ along $\alpha$ induced by a trivialization of $\bar u^*\xi$.
Since  $\wind_{\infty}(\bar{u},\infty) \leq \frac{\CZ(\alpha)}{2}=1,$
we conclude from \eqref{windpi} that
$0 \leq {\rm wind}_{\pi}(\bar u)\leq 1-1=0$, which implies 
$
\wind_{\pi}(\bar{u})=0.
$
In particular, $\bar u$ is an immersion transverse to $R_\lambda$.

Since  $1=\ind(u) > 2g -2 + \#\Gamma_0(u) = -1$, Lemma \ref{lem4.1HN} implies that  $D_u:\Gamma(N) \to \Gamma(F)$ is surjective, and hence $\M_{0,1}^J({\alpha}; \emptyset)$ is a one-dimensional manifold. In particular, $\M_{0,1}^J({\alpha}; \emptyset)/\R$ is a zero-dimensional manifold.

It remains to show that $\M_{0,1}^J({\alpha}; \emptyset)$ is compact.
To see this, let $\{u_n\}$ be a sequence in $\M^J_{0,1}(\alpha; \emptyset)$. We know from the SFT-compactness theorem that there exists a subsequence of $\{u_n\}$, also denoted  $\{u_n\}$, that converges to an index-$1$ building $v=(v_1,\ldots,v_m)$, with a positive end at $\alpha$, and no negative ends. Theorem \ref{B} implies that  $v=v_1$ is a plane asymptotic to $\alpha$, from which we conclude that $v \in \M_{0,1}^J(\alpha;\emptyset)$. In particular $\{u_n\}$ is eventually constant equal to $v$ and thus $\M_{0,1}^J(\alpha;\emptyset)/\R$ is a compact zero-dimensional manifold.
\end{proof}

\begin{theorem}[Wendl, {\cite[Proposition 2.2]{W10}}] \label{teo1Wendl}
Let $\lambda$ be a nondegenerate contact form on a closed $3$-manifold $M$, and let  $J \in \mathcal{J}(\lambda)$. Let $\Sigma$ be a closed Riemann surface, $\Gamma \subset \Sigma$ a finite set, and $u:\Sigma \setminus \Gamma \to \R \times M$  an immersed  $J$-holomorphic curve with positive ends at the Reeb orbits $\alpha_1,\ldots,\alpha_n$ and negative ends at the Reeb orbits $\beta_1,\ldots,\beta_m$. Define
\begin{equation*}
    2c_N(u):=\ind(u)- 2 + 2g + \#\Gamma_0,
\end{equation*}
where $\#\Gamma_0$ is the number of even punctures of $u$, and let
$$
K(a,b):= \min \{k_1+2k_2 ; k_1,k_2 \in \Z^{\geq 0}, k_1 \leq b \mbox{ and } k_1+k_2 > a\}.$$
The following assertions hold:
\begin{itemize}
 \item[(i)] If $\ind(u) \leq 0$, then   
\begin{equation*}
    0 \leq \dim \ker D_u \leq K(c_N(u),\#\Gamma_0(u)).
\end{equation*}
\item[(ii)] If $\ind(u) \geq 0$, then 
$$
\ind(u) \leq \dim \ker D_u \leq \ind(u) +K(c_N(u)-\ind(u), \#\Gamma_0(u)).
$$
\end{itemize}
\end{theorem}

\begin{theorem}[Hutchings, {\cite[Teorema 4.1]{hutchings2009}}] \label{HTteo4.1}
Let $\lambda$ be a contact form on a closed $3$-manifold $M$. There is a residual subset $\mathcal{J}_{\rm reg}(\lambda) \subset \mathcal{J}(\lambda)$ such that if $J \in \mathcal{J}_{\rm reg}(\lambda)$ and $u$ is a simple $J$-holomorphic curve with $\ind(u)\leq 2$, then $u$ is an embedding.
\end{theorem}

\begin{definition}
Let $\lambda$ be a contact form on a closed $3$-manifold $M$,  and let $\pi:TM \to \xi$ be the projection along the Reeb vector field $R_{\lambda}$. Let $J \in \mathcal{J}(\lambda)$, and let $u:\Sigma \setminus \Gamma \to \R \times M$ be a $J$-holomorphic curve. Define by $\windd_{\pi}(u)$ the sum of the local indices of the zeros of $\pi \circ du$, and let
\begin{equation*}
      \windd_{\infty}(u)=\sum_{z \in \Gamma^{+}} \wind_{\infty}(u;z)-\sum_{z \in \Gamma^{- }} \wind_{\infty}(u;z).
\end{equation*}
where $\wind_{\infty}(u;z)$ is  defined as in \eqref{eq:windinfty}. As the notation suggests, $\windd_\infty(u)$ does not depend on the trivialization $\tau$.
\end{definition}

\begin{prop}[{\cite[Proposição 5.6]{prop2}}] \label{prop5.6HWZ}
Let $\lambda$ be a nondegenerate contact form on a closed $3$-manifold $M$, and let $J \in \mathcal{J}(\lambda)$. Let $u:\Sigma \backslash \Gamma \to \R \times M$ be a $J$-holomorphic curve. Then
$\windd_{\pi}(u)=\wind_{\infty}(u)-\chi(\Sigma)+ \sharp \Gamma.$
\end{prop}

\begin{lemma} \label{partialainN}
Let $\lambda$ be a nondegenerate contact form on a closed $3$-manifold $M$, and let $J \in \mathcal{J}( \lambda)$. Let $\Sigma$ be a genus zero closed Riemann surface, and let $v_1:\dot \Sigma = \Sigma \setminus \Gamma \to \R \times M$ be the first level of a building as in Theorem \ref{C}-(iii)-\ref{itm:vi}, that is,
$v_1$ is a $J$-holomorphic curve with Fredholm index $2-d$, a positive end at $\alpha_1^d$ and $d+1$ negative ends, one at $\alpha_2^d$ and the other $d$ ends at the same index-$2$ Reeb orbit $\alpha_3$. Assume that $d>1$. Then  $\partial_a \notin dv_1(z) \cdot T_z\dot \Sigma, \forall z\in \dot \Sigma$. 
\end{lemma}
\begin{proof}
Let $\Gamma=\{z_1,z_2,z_3^1,\ldots,z_3^d\}$ be the set of punctures of $v_1$ so that 
 $v_1$ is asymptotic to $\alpha_1^d$ at $z_1$, $\alpha_2^d$ at $z_2$, and $\alpha_3$ at $z_3^i$ for all $1 \leq i \leq d$. Since $\CZ(\alpha_3)=2$, we have $\wind_{\infty}(v_1;z_3^i)\geq 1, \forall i.$   Proposition \ref{prop5.6HWZ} gives
$$
\begin{aligned}
0\leq \windd_{\pi}(v_1) &  \leq  \wind_{\infty}(v_1;z_1) - \wind_{\infty}(v_1;z_2) -d-2 +(d+2)\\
& = \wind_{\infty}(v_1;z_1) - \wind_{\infty}(v_1;z_2).
\end{aligned}
$$
Since $d>1$, we conclude from Lemma \ref{odepeq} that  $\alpha_1^d$ and $\alpha_2^d$ are necessarily elliptic. In particular,
$$\wind_{\infty}(v_1;z_1) \leq \frac{\CZ(\alpha_1^d)-1}{2} \;\; \mbox{ and } \;\; \wind_{\infty}(v_1;z_2) \geq \frac{\CZ(\alpha_1^d)+1}{2}.$$

Since $\ind(v_1)=2-d$, we have $\CZ(\alpha_1^d)-\CZ(\alpha_2^d)=2$. Hence
$$
\begin{aligned}
0 &\leq \windd_{\pi}(v_1)=\wind_{\infty}(v_1;z_1) - \wind_{\infty}(v_1;z_2)\\
&\leq  \frac{\CZ(\alpha_1^d)-1}{2} - \frac{\CZ(\alpha_2^d)+1}{2} \\
&= \frac{\CZ(\alpha_1^d)-\CZ(\alpha_2^d)}{2}-1\\
&= \frac{2}{2}-1=0.
\end{aligned}
$$
We conclude that ${\rm wind}_{\pi}(v_1)=0$, i.e., $\pi \circ dv_1$ never vanishes. 
\end{proof}


\begin{lemma} \label{dim1} 
Assume that hypotheses of Lemma \ref{partialainN} hold. Assume, moreover, that $J \in \mathcal{J}_{\text{reg}}(\lambda),$ where $\mathcal{J}_{\text{reg}}(\lambda)$ is given in Theorem \ref{HTteo4.1}. Consider the  operator 
$
D_{v_1}^N: W^{1,p}(N_{v_1}) \to  L^p(\overline{\rm Hom}_{\C}(T\Sigma, N_{v_1})),$
defined as in section \ref{sec:basics}.  Then
 $\dim \ker D_{v_1}^N = 1$. 
\end{lemma}
\begin{proof}

 Let $\mathfrak{v}_1$ be the somewhere injective curve $d$-covered by $v_1$.  Since $\mathfrak{v}_1$ has index $1$, we conclude  from Theorem \ref{HTteo4.1} that $\mathfrak{v}_1$ is an embedding. Since the covering $v_1 \to \mathfrak{v}_1$ has no branching point, $v_1$ is an immersion.

Consider a trivialization $\tau$ of $v_1^*\xi$ that extends to a trivialization of $\xi$ along a disk that bounds $\alpha_3$. Recall that $\ind(v_1)=2-d$ implies  $\CZ(\alpha_1^d)-\CZ(\alpha_2^d)=2$. Since $d>1$, it follows from Lemma \ref{odepeq} that $\sharp \Gamma_0(v_1)=d$, that is $\alpha_1$ and $\alpha_2$ are elliptic. Therefore, we compute
 $$
 c(v_1,\Gamma)  = c_{\tau}(v_1)-\chi(v_1) +\frac{1}{2}(\CZ(\alpha_1^d)-\CZ(\alpha_2^d)-d\CZ(\alpha_3)-\sharp \Gamma_1(v_1))=0
 $$
 By Theorem \ref{teo1Wendl}, we obtain
  $$
  \dim \ker D_{v_1}^N \leq \min \big\{k_1+k_2 \; ; \; k_1,k_2\in \Z^{\geq 0}, k_1 \leq \#\Gamma_0(u)=d \mbox{ and } 2k_1+k_2> 2 c(v_1,\Gamma)=0\big\}=1,
  $$ where the minimum above is realized for $k_1=1$ and $k_2=0$.

  Finally, we know that $\dim \ker D_{v_1}^N\geq 1$ since the $\R$-translations of $v_1$ determines a non-trivial section of $N_{v_1}$.  \end{proof}

\section{Gluing of buildings}\label{sec:gluing}

The goal of this section is to analyze buildings of type (iii.6) in Theorem \ref{C} and show that they do not contribute to the boundary of the moduli space of index-2 cylinders. This is achieved via gluing constructions and an analysis of orientations, which shows that such contributions cancel in pairs.
Our construction follows the obstruction bundle gluing theory of
Hutchings-Taubes~\cite{hutchings2009}, and we also refer to Bao-Honda \cite{BH18} for a similar approach. See also \cite{datta2025invitation}. Instead of reviewing all the analytic
machinery, we recall only the ingredients that are specific to our situation and explain in detail the only additional technical point, namely the non–vanishing gluing condition $\bar x_k \neq 0 \forall k$,
which holds for generic $J$.  
If $\bar{x}_k = 0$ for some $k$, then the corresponding building develops an additional breaking and has at least three levels. Such configurations lie in strata of codimension at least two and therefore do not appear as boundary points of one-dimensional moduli spaces for generic $J$.

The gluing analysis for these buildings is based on a finite-dimensional obstruction bundle construction. More precisely, the upper level consists of a multiply covered curve of Fredholm index $2-d$, whose linearized normal Cauchy–Riemann operator has a $(d-1)$–dimensional cokernel. The lower level consists of $d$ rigid holomorphic planes, each contributing a gluing parameter.
The pregluing construction produces approximate solutions depending on $(d+1)$ translation parameters. The failure of these approximate solutions to be holomorphic defines an error term, whose projection onto the cokernel of the linearized operator at the upper level gives rise to a finite-dimensional obstruction map $\mathfrak s$. After quotienting by the $\mathbb{R}$–translation, the space of gluing parameters has dimension $d$, while the obstruction space has dimension $d-1$, so that the resulting moduli space is one-dimensional. 

Throughout this section, we consider only holomorphic buildings $\mathcal B$ that arise as limits of 
one-dimensional moduli spaces of index-2 holomorphic cylinders, for generic choices of 
$J \in \mathcal{J}_{\mathrm{reg}}(\lambda)$. For generic J, somewhere injective curves are Fredholm regular, see 
\cite{W10}. In the presence of symmetry, such as for multiply covered 
curves, transversality becomes more subtle, see \cite{Zhou_polyfold}. In our setting, 
these non-regular configurations are handled using obstruction bundle 
techniques.

Let
$\mathcal B = (v_1, v_2)$
be a building as in Theorem \ref{C}-(iii.6). The top level $v_1$ is a $d$–fold cover of a somewhere injective pair of
pants with one positive end at $\alpha_1^d$, one negative end at $\alpha_2^d$, and $d$ negative ends at the index-$2$ orbit $\alpha_3$. Its index is ${\rm ind}(v_1)=2-d$ . The lower level consists of a trivial cylinder over $\alpha_2^d$ together with $d$
index-$1$ planes $u_1,\ldots u_d$ asymptotic to $\alpha_3$.
Hence, all lower curves are regular,
 the only non–regular component is $v_1$, and the total index of the building equals $2$. Consequently, the obstruction space comes entirely from the top level $v_1$.

\begin{figure}[ht]
\centering 
\includegraphics[scale=0.4]{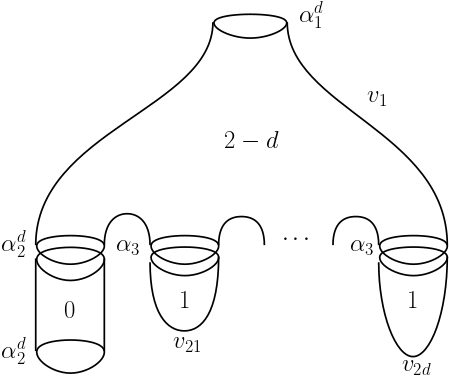} 
\caption{J-holomorphic building $\mathcal B= (v_1,v_2)$}
\label{figura:c4}
\end{figure}

The parameters in the gluing construction correspond to
translations of the curves in the $\R$-direction.  Let
$T=(T_0,T_1,\dots,T_d)\in(0,+\infty)^{d+1}
$
where $T_0$ translates the top curve $v_1$ upward and $T_k$
translates the plane $v_{2k}, k=1,\ldots,d,$ downward. 
After applying these translations, we obtain curves
$
v_{1,T_0}, v_{21,T_1}, \ldots, v_{2d,T_d}.
$
After quotienting by the overall $\mathbb{R}$–translation, the effective gluing parameter space has dimension $d$. On the other hand, since the upper level has Fredholm index $2-d$ and its kernel is generated by the $\mathbb{R}$–translation (see Lemma \ref{dim1}), the cokernel has dimension $d-1$. Therefore, the zero set of the obstruction map $\mathfrak{s}$ is expected to be one-dimensional, corresponding to ends of the moduli space of index-2 cylinders.

Near each puncture, we consider cylindrical coordinates
$(s,t)$ where the curves converge to $\alpha_3$.
More precisely, after choosing a trivialization of the normal bundle
of $\alpha_3$, the curves can be written as graphs
$
v_{1,T_0}(s,t)
=\exp_{\alpha_3(t)}(\eta_{1k}(s,t))$ and $
v_{2k,T_k}(s,t)
=\exp_{\alpha_3(t)}(\eta_{2k}(s,t)),
$
where the sections $\eta_{1k}$ and $\eta_{2k}$ along the upper and lower parts, respectively, decay exponentially as
$s\to -\infty$ and $s\to +\infty$, respectively.

Let
$R \gg \max_{0\leq k\leq d} T_k$
be the gluing parameter.
After translating the curves by $T\in (0,+\infty)^{d+1}$,
the relevant part of $v_{1,T_0}$ is centered near $s=R/2$,
while the relevant part of $v_{2k,T_k}$ is centered near $s=-R/2$.
Choose a smooth function
$\beta:\R\to[0,1]$
such that
$
\beta(s)=
1$ for  $s\le -1$ and $\beta(s)=0$ for $s\ge 1$.
Define the translated cutoff functions $\beta_1(s)=\beta(s-R/2)$ and $\beta_{2k}(s)=1-\beta(s+R/2), k=1,\ldots,d.$
On the neck
$
[-R/2,R/2]\times S^1$
corresponding to the $k$-th puncture we define
$$v_{T,R}(s,t)
=\exp_{\alpha_3(t)}
\Big(\beta_1(s)\,\eta_{1k}(s-R/2,t) +\beta_{2k}(s)\,\eta_{2k}(s+R/2,t) \Big).
$$

Outside the neck regions, the map coincides with
$v_{1,T_0}$ and $v_{2k,T_k}$.
The resulting domain $\Sigma_{T,R}$ is obtained by inserting
cylinders of length $R$ at the punctures. Standard asymptotic estimates imply that the preglued map $v_{T,R}:\Sigma_{T,R} \to \R \times M$ satisfies
$\|\bar\partial_J v_{T,R}\| = O(e^{-\delta R})$
for some $\delta>0$.
Because the curves $v_1,v_{21},\ldots,v_{2k}$ are $J$-holomorphic, the region where $v_{T,R}$ may not satisfy the Cauchy-Riemann equation is the support of $\beta_i'$.

We now look for a section $\psi$ of the normal bundle of $v_{T,R}$ which has the form
$$
\psi=\beta_1\psi_1+\sum_{k=1}^d \beta_{2k}\psi_{2k},
$$
where
$\psi_1$ and $\psi_{2k}$ are sections of the normal bundle of $v_1$ and $v_{2k}$, respectively. 
We then define the perturbed curve by
$$
v=\exp_{v_{T,R}}(\psi).
$$
The condition that $v$ be $J$--holomorphic is $\bar\partial_J(v)=0.$
Expanding near $v_{T,R}$, we obtain
$$
\bar\partial_J(v)
=
\bar\partial_J(v_{T,R})
+
Dv_{T,R}\psi
+
R(\psi),$$
where $Dv_{T,R}$ is the linearization of the Cauchy--Riemann operator at $v_{T,R}$ and $R(\psi)$ contains  nonlinear terms, satisfying
$$
R(\psi)=O(|\psi|\,|\nabla\psi|)+O(|\psi|^2).
$$

Using the decomposition of $\psi$, the equation
$\bar\partial_J(v)=0$ can be written as
$$
\beta_1\Theta_1(\psi_1,\psi_{21},\dots,\psi_{2d})
+
\sum_{k=1}^d \beta_{2k}\Theta_{2k}(\psi_1,\psi_{2k})
=0,
$$
where $\Theta_1$ and $\Theta_{2k}$ are the components of the nonlinear
Cauchy--Riemann equation on the upper and lower parts, respectively.
More precisely, the equation on the top level has the form
\begin{align*}
\Theta_1(\psi_1,\psi_{21},\dots,\psi_{2d})
&=
Dv_1\psi_1+F(\psi_1)
+\frac12\sum_{k=1}^d (\partial_s\beta_{2k})(\eta_{2k}+\psi_{2k})\,d\bar z \\
&\quad
+\frac12\sum_{k=1}^d
\Big[
(\nabla_{\psi_1}J)(\partial_s\beta_{2k})\psi_{2k}\,ds
+\beta_{2k}(\nabla_{\psi_1}J)\nabla\psi_{2k}
+2\widetilde R_{1k}(\psi_1,\psi_{2k})
\Big],
\end{align*}
and for the lower components, we have
\begin{align*}
\Theta_{2k}(\psi_1,\psi_{2k})
&=
Dv_{2k}\psi_{2k}+F_k(\psi_{2k})
+\frac12(\partial_s\beta_1)(\eta_{1k}+\psi_1)\,d\bar z \\
&\quad
+\frac12
\Big[
(\nabla_{\psi_{2k}}J)(\partial_s\beta_1)\psi_1\,ds
+\beta_1(\nabla_{\psi_{2k}}J)\nabla\psi_1
\Big]
+\widetilde R_{2k}(\psi_1,\psi_{2k}),
\end{align*}
for every $k=1,\ldots, d$.

Here, $Dv_1=D_{v_1}^N$ and $Dv_{2k}=D_{v_{2k}}^N$ are the normal linearized Cauchy-Riemann operators
along $v_1$ and $v_{2k}$, $F(\psi_1)$ and $F_k(\psi_{2k})$ are quadratic terms of type 1,
and $\widetilde R_{1k}$ and $\widetilde R_{2k}$ contain the remaining nonlinear
terms.

Therefore, solving the nonlinear gluing problem is equivalent to solving the
system
\[
\Theta_1(\psi_1,\psi_{21},\dots,\psi_{2d})=0,
\qquad
\Theta_{2k}(\psi_1,\psi_{2k})=0,\quad k=1,\dots,d.
\]

The equations $\Theta_{2k}=0, k=1,\ldots, d,$
can be solved first, because the operators $Dv_{2k}$ are surjective.  More precisely, if $Q_{2k}$ denotes a bounded right inverse of
$Dv_{2k}$, then the equation $\Theta_{2k}=0$ can be rewritten as
$$
\begin{aligned}\psi_{2k}
& =
-\,Q_{2k}\Big(
F_k(\psi_{2k})
+\frac12(\partial_s\beta_1)(\eta_{1k}+\psi_1)\,d\bar z
+\frac12\big[(\nabla_{\psi_{2k}}J)(\partial_s\beta_1)\psi_1\,ds\\ &
+\beta_1(\nabla_{\psi_{2k}}J)\nabla\psi_1\big]
+\widetilde R_{2k}(\psi_1,\psi_{2k})
\Big].
\end{aligned}$$
For $R\gg |T|\gg 0$ and $\psi_1$ small, this is a contraction mapping equation,
and thus has a unique solution
$
\psi_{2k}=\psi_{2k}(\psi_1,T), k=1,\dots,d.
$

Substituting these solutions into the equation $\Theta_1=0$ gives an equation
for $\psi_1$
$$
\Theta_1\bigl(\psi_1,\psi_{21}(\psi_1,T),\dots,\psi_{2d}(\psi_1,T)\bigr)=0.
$$

Since $Dv_1$ is not necessarily surjective, we consider the decomposition 
$
L^p(\Lambda^{0,1}T^*\dot\Sigma_1\otimes N_{v_1})={\rm image}(Dv_1)\oplus \ker(Dv_1^*).
$ 
Let
$$
\Pi_+:L^p(\Lambda^{0,1}T^*\dot\Sigma_1\otimes N_{v_1})\to {\rm image}(Dv_1),
\qquad
\Pi_{W_+}:L^p(\Lambda^{0,1}T^*\dot\Sigma_1\otimes N_{v_1})\to \ker(Dv_1^*)
$$
denote the corresponding projections.

The equation $\Theta_1=0$ is therefore equivalent to the system
\begin{align}
\Pi_+\Theta_1\bigl(\psi_1,\psi_{21}(\psi_1,T),\dots,\psi_{2d}(\psi_1,T)\bigr)&=0,
\label{eq:proj-image}\\
\Pi_{W_+}\Theta_1\bigl(\psi_1,\psi_{21}(\psi_1,T),\dots,\psi_{2d}(\psi_1,T)\bigr)&=0.
\label{eq:proj-cokernel}
\end{align}

The first equation determines the component of $\psi_1$ orthogonal to the
cokernel. Indeed, if $Q_1$ is a bounded right inverse of $Dv_1$ on
${\rm image}(Dv_1)$, then \eqref{eq:proj-image} can be rewritten as
$$
\begin{aligned}
    \psi_1
& =
-\,Q_1\Pi_+\Big(
F(\psi_1)
+\frac12\sum_{k=1}^d (\partial_s\beta_{2k})(\eta_{2k}+\psi_{2k}(\psi_1,T))\,d\bar z\\
& 
+\frac12\sum_{k=1}^d
\Big[
(\nabla_{\psi_1}J)(\partial_s\beta_{2k})\psi_{2k}(\psi_1,T)\,ds
+\beta_{2k}(\nabla_{\psi_1}J)\nabla\psi_{2k}(\psi_1,T)
\\& +2\widetilde R_{1k}\bigl(\psi_1,\psi_{2k}(\psi_1,T)\bigr)
\Big]
\Big).
\end{aligned}
$$
Again, for $|T|$ large this is solved by contraction, giving a unique
small solution
$
\psi_1=\psi_1(T).$

We are thus left with the finite-dimensional equation obtained by projecting
to the cokernel
$$
\Pi_{W_+}\Theta_1\bigl(\psi_1(T),\psi_{21}(T),\dots,\psi_{2d}(T)\bigr)=0.
$$

This equation defines the obstruction section
$$
\mathfrak{s}:(0,+\infty)^{d+1}\to W_+,
\qquad
\mathfrak{s}(T_0,\dots,T_d)
:=
\Pi_{W_+}\Theta_1\bigl(\psi_1(T),\psi_{21}(T),\dots,\psi_{2d}(T)\bigr).
$$
The zeros of this map correspond to genuine holomorphic cylinders obtained by gluing.
Equivalently, we have
\begin{align*}
\mathfrak{s}(T_0,\dots,T_d)
&=
\Pi_{W_+}\Bigg[
Dv_1\psi_1(T)+F(\psi_1(T))
+\frac12\sum_{k=1}^d (\partial_s\beta_{2k})(\eta_{2k}+\psi_{2k}(T))\,d\bar z \\
&
+\frac12\sum_{k=1}^d
\Big(
(\nabla_{\psi_1(T)}J)(\partial_s\beta_{2k})\psi_{2k}(T)\,ds\\ & 
+\beta_{2k}(\nabla_{\psi_1(T)}J)\nabla\psi_{2k}(T)
+2\widetilde R_{1k}(\psi_1(T),\psi_{2k}(T))
\Big)
\Bigg].
\end{align*}

Since $Dv_1\psi_1(T)\in {\rm image}(Dv_1)$, its projection to $W_+$ vanishes, and,
therefore, the nonlinear obstruction equation becomes
\begin{align}
\mathfrak{s}(T_0,\dots,T_d)
&=
\Pi_{W_+}\Bigg[
F(\psi_1(T))
+\frac12\sum_{k=1}^d (\partial_s\beta_{2k})(\eta_{2k}+\psi_{2k}(T))\,d\bar z
\notag\\
&
+\frac12\sum_{k=1}^d
\Big(
(\nabla_{\psi_1(T)}J)(\partial_s\beta_{2k})\psi_{2k}(T)\,ds
+\beta_{2k}(\nabla_{\psi_1(T)}J)\nabla\psi_{2k}(T)\\ & 
+2\widetilde R_{1k}(\psi_1(T),\psi_{2k}(T))
\Big)
\Bigg]
=0.
\label{eq:obstruction-section}
\end{align}

We define the obstruction bundle by
$$
\mathcal O := (R_0,+\infty)^{d+1}\times W_+ \longrightarrow (R_0,+\infty)^{d+1}, \qquad R_0 \gg 0.
$$
The map $\mathfrak{s}$ is a section of $\mathcal O$, and its zeros are in one-to-one
correspondence with solutions of the nonlinear gluing problem. Thus, the gluing problem has been reduced to the finite-dimensional nonlinear
equation \eqref{eq:obstruction-section}.

To make the dependence on $T=(T_0,T_1,\ldots,T_d)$ more explicit, recall that near the $k$-th negative puncture of $v_1$ the pregluing error has a leading term coming from the lower level
$$
\eta_{2k}(s,t)=c_k e^{\mu (s+T_k)}\phi(t)+h_k(s,t),
$$
where $\mu<0$ is the largest negative eigenvalue of the asymptotic operator
at $\alpha_3$, $\phi$ is a unit $\mu$-eigensection, $c_k\in \{-1,+1\}$, and
$h_k(s,t)\to 0$ exponentially faster than $e^{\mu s}$.  Hence, the leading contribution comes from the terms
$$
\frac12 \sum_{k=1}^d (\partial_s\beta_{2k})\,\eta_{2k}\,d\bar z .
$$
Also, for $\sigma\in \ker Dv_1^*$ we write its asymptotic expansion near
the $k$--th negative puncture in the form
$$
\sigma_k(s,t)=e^{-\mu s}\,b_k(\sigma)\,\phi(t)+\widetilde h_k(s,t),
$$
where $b_k(\sigma)\in \R$ and $\widetilde h_k(s,t)\to 0$ exponentially faster than $e^{-\mu s}$ as $s\to -\infty$.

Recall that the nonlinear equation on the top level can be written as
$$
F_1(\psi_1,T)
:=
\Theta_1\bigl(\psi_1,\psi_{21}(\psi_1,T),\dots,\psi_{2d}(\psi_1,T)\bigr).
$$

Using the asymptotic expansion
for $\eta_{2k}$, we separate the dominant term in the pregluing error and define
$$
\widetilde F_1(T)
:=
F_1(\psi_1(T),T)
-
\frac12
\sum_{k=1}^d
(\partial_s\beta_{2k})
\,c_k e^{\mu (s+T_k)}
\phi(t)\, d\bar z.
$$

Thus $\widetilde F_1(T)$ contains all the remaining contributions to the
nonlinear equation on the top level. Substituting these asymptotic expansions into the pairing defining
$\mathfrak{s}$, we obtain
\begin{align*}
\mathfrak{s}(T_0,\dots,T_d)(\sigma)
&=
\left\langle \sigma,\,
\frac12 \sum_{k=1}^d (\partial_s\beta_{2k})\,\eta_{2k}\,d\bar z
\right\rangle
+
\left\langle \sigma,\,
\widetilde F_1(T)
\right\rangle \\
&=
\frac12 \sum_{k=1}^d b_k(\sigma)c_k e^{\mu T_k}
+
\mathcal{R}(T)(\sigma),
\end{align*}
where $\mathcal R(T)(\sigma) := \langle \sigma, \widetilde F_1(T) \rangle$ is the remainder term.

Thus, the obstruction equation can be written as
\begin{equation}\label{eq:obstruction-explicit}
\mathfrak{s}(T_0,\dots,T_d)(\sigma)
=
\frac12 \sum_{k=1}^d b_k(\sigma)c_k e^{\mu T_k}
+
\mathcal{R}(T)(\sigma)
=0,
\qquad
\forall \sigma\in \ker D_1^*.
\end{equation}

We now describe explicitly the finite–dimensional reduction that leads
to a linear system $Bx=b$ and explain why its solution $\bar x=B^{-1}b$ satisfies
$\bar x_k \neq 0$ for generic $J$.

Choose a basis
$\sigma_1,\dots,\sigma_{d-1}$
of $\ker Dv_1^*$. 
At the $k$–th negative end these sections satisfy the asymptotic
expansion
$$
\sigma_j(s,t)
=
b_{j,k}\,e^{-\mu s}\phi(t)
+
o(e^{-\mu s}),
\qquad s\to-\infty,
$$
where $b_{j,k}\in\R$ are the corresponding leading coefficients of $\sigma_j$ at the $k$-th end. We may assume that $b_{j,1}=1$ for every $j$, see Proposition \ref{kerDu2}.
 
Recall that the parameter $T_0$ translates the top level $v_1$ in the
$\R$–direction. Its role is to push the upper level sufficiently far from the
lower planes so that the corresponding error terms become exponentially small.
In particular, the contribution coming from the upper level produces a term of
order $O(e^{\mu T_0})$ in the obstruction equation. Since $\mu<0$, this term
decays when $T_0\to+\infty$ and will be negligible compared with the leading
contribution coming from the lower levels.

Taking the $L^2$–pairing and keeping only the leading contribution, we obtain
\begin{equation}\label{obstruction_non_linear}
\mathfrak s(T)(\sigma_j)
=
O(e^{\mu T_0})+ \sum_{k=1}^{d} b_{j,k}\, c_k\, e^{\mu T_k}
+
\text{(higher order terms)}, \qquad  j=1,\ldots, d-1.
\end{equation}

Using $T_1\gg 1$ as the free parameter, we normalize by
$c_1e^{\mu T_1}$ and define $x:=(x_1,\ldots,x_{d-1})$ by
$$
x_{k}:=\frac{c_{k+1}}{c_1}e^{\mu(T_{k+1}-T_1)}
,\qquad k=1,\dots,d-1.
$$
Thus $e^{\mu(T_{k+1}-T_1)}>0$ and the sign of $x_{k}$ is exactly the 
ratio $c_{k+1}/c_1\in\{\pm1\}$ coming from the asymptotic leading coefficients
of the planes in $v_2$.

Using that $\mathfrak s(T)(\sigma_j)=0$ for all $j$, we obtain a system of equations of the form
\begin{equation}\label{non-linear_x}
Bx-b + R_{T_0,T_1}(x)=0,
\end{equation}
where 
$$
B=(b_{j,k+1})\in M_{d-1}(\R), \qquad b=(-1)\in\R^{d-1},$$ 
and
$R_{T_0,T_1}(x)$ contains higher order terms. Here, one may write
$$
R_{T_0,T_1}(x)=G_{T_1}(x)+E_{T_0,T_1},
\qquad
\|E_{T_0,T_1}\|\le C e^{-\delta_0 (T_0-T_1)}, \quad \delta_0>0,
$$
where $G_{T_1}(x)$ depends only on the relative neck data and the local
gluing near the ends.

We first consider the linear system $Bx=b$ and show that $B$ is invertible.
If $B$ were singular, there would exist a nontrivial linear combination
$\sigma=\sum_j a_j\sigma_j$ whose leading coefficients vanish at all
negative ends.
Such a section would decay faster than the leading exponential rate at
every end, which forces $\sigma\equiv0$, see Proposition \ref{kerDu2}.
This contradicts linear independence of the $\sigma_j$. Therefore, the linear system $Bx=b$ admits a unique solution
$\bar x = B^{-1} b$.

To solve the non-linear system, we first observe that there exists $\delta_1>0$ and a neighborhood $U$ of the linear solution
$\bar x=B^{-1}b$ such that
$$
\|G_{T_1}\|_{C^1(U)} \le C e^{-\delta_1 T_1}.
$$
In particular, for $T_0\gg T_1\gg 1$ we have
$$
\|R_{T_0,T_1}\|_{C^1(U)} \le C\big(e^{-\delta_1 T_1}+e^{-\delta_0 (T_0-T_1)}\big).
$$

Since $B\in GL_{d-1}(\R)$, for $T_0\gg T_1\gg 1$ the map
$x\mapsto Bx-b+R_{T_0,T_1}(x)$ is a $C^1$-small perturbation of the
isomorphism $x\mapsto Bx-b$ on $U$.
Hence, by the implicit function theorem, there exists a unique solution
$x=x(T_0, T_1)\in U$ of the non-linear system. Moreover, $x(T_0,T_1)\to\bar x$ as
$T_1\to\infty$ and $T_0-T_1\to\infty$.

 Once $x=(x_1,\ldots, x_{d-1})$ is determined, the remaining neck lengths are recovered as
$$
T_{k+1} = T_1 + \frac{1}{\mu}\log \left(\frac{c_1}{c_{k+1}}x_k\right),
\qquad k=1,\dots,d-1,
$$
so $T_2,\dots,T_d$ are functions of $T_1$.  Hence, for $T_{k+1}$ to be well-defined, we must have
$
\frac{c_1}{c_{k+1}} x_k > 0,
$
and in particular $x_k \neq 0$ for all $k$.

We now show that this condition holds for generic $J$. Recall that
\[
x_k = \frac{c_{k+1}}{c_1} e^{\mu(T_{k+1}-T_1)}.
\]
Thus, if $x_k = 0$, then necessarily $T_{k+1}-T_1 \to +\infty$, and
the $(k+1)$-th neck becomes infinitely longer than the others.
This produces a separation of scales in the neck region, and the
preglued curves develop an additional breaking. Standard SFT
compactness then implies that the limiting building has at least three
levels.
On the other hand, the obstruction bundle gluing for $\mathcal B$
produces a $1$--dimensional family of cylinders, parametrized by the
neck lengths $T_1,\dots,T_d$ subject to $d-1$ independent obstruction
equations.
If such an additional breaking occurs, then at least one new independent
neck parameter is introduced, corresponding to the new interface between
levels, while the obstruction coming from the multiply covered top level
persists. Therefore, the resulting configurations are more constrained
than in the two-level case.

Equivalently, from the point of view of SFT compactness, each additional
breaking produces a higher codimension stratum in the compactified moduli space. Two-level buildings form the codimension-one boundary of the $1$-dimensional moduli space, whereas buildings with at least three levels lie in strata of codimension at least two. Consequently, such configurations do not arise generically as boundary
points of the $1$--dimensional glued moduli space. This is consistent with the general structure of SFT compactifications, where each additional level increases the codimension by at least one.

We conclude that for generic $J$, we have $\bar x_k \neq 0$.
Since $x(T_0,T_1) \to \bar x$ as $T_1 \to \infty$ and $T_0 - T_1 \to
\infty$, it follows that $x_k(T_0,T_1)$ remains uniformly bounded away
from zero for $T_0 \gg T_1 \gg 1$.  Switching the planes on the second level with the `opposite' planes, the condition $x_k \neq 0 \forall k,$ implies that we can also glue the building with the opposite planes and obtain another family of index-$2$ $J$-holomorphic cylinders with a positive end at $\alpha_1^d$ and a negative end at $\alpha_2^d$. As we prove in the next section, these two families will cancel out each other in the proof of $\partial^J \circ \partial^J =0=0$.

The discussion above proves the following theorem.

\begin{theorem}
\label{thm:gluing-opposite}
For generic $J$, let $v_1$ be a top curve with $d$ negative ends asymptotic to $\alpha_3$, as in the index-$2$ buildings of Theorem \ref{C}-(iii.6), lying in the end of $\mathcal M^J_{0,2}(\alpha_1^d ; \alpha_2^d)$.   Assume that for each such end there exists a pair of index-$1$ rigid planes $w_i, w_i'$ asymptotic to $\alpha_3$ through opposite directions.
Then there exist choices of $v_{2k}\in \{w_k,w_k'\},k=1,\ldots,d,$ so that the building $\mathcal B=(v_1,v_2)$ whose second level $v_2$ consists of a trivial cylinder over $\alpha_2^d$ and the $d$ planes $v_{21},\ldots,v_{2d},$ can be glued to a $1$-parameter family $\mathcal C$ of 
$J$-holomorphic cylinders in $\mathcal M^J_{0,2}(\alpha_1^d ; \alpha_2^d)$.
Moreover, the building 
 $\mathcal B'=(v_1,v_2')$ obtained by replacing $v_{2k}$ with the opposite plane $v_{2k}'\in \{w_k,w_k'\}$, for every $k=1,\ldots, d,$ can also be glued to a 
$1$-parameter family $\mathcal C'$ of $J$-holomorphic cylinders in $\mathcal M^J_{0,2}(\alpha_1^d ; \alpha_2^d)$.
\end{theorem}

\begin{rem}
The choice of $v_{21},\dots,v_{2d}$ in Theorem \ref{thm:gluing-opposite} is made as follows.
Fix any $v_{21} \in \{w_1,w'_1\}$ and let
$\bar x=(\bar x_1,\dots,\bar x_{d-1})
$ be the unique solution of the linear system $Bx=b$, where $B$
is determined by the asymptotics of the top curve $v_1$ and a
chosen normalized basis of $\ker Dv_1^*$. For generic $J$, we have $\bar x_k \neq 0$ for all $k$.
Since $x(T_0,T_1)\to \bar x$ as $T_1\to\infty$ and
$T_0-T_1\to\infty$, it follows that for $T_0\gg T_1\gg1$
the signs of the components $x_k$ coincide with the signs of
$\bar x_k$. For each $k=2,\dots,d$, choose $v_{2k}\in\{w_k,w'_k\}$ so that the
corresponding coefficient $c_k$ satisfies $\frac{c_k}{c_1}\,\bar x_{k-1}>0.$
Then, the sign of $x_{k-1}$ agrees with the required ratio
$c_k/c_1$, and the neck lengths
$$
T_k = T_1 + \frac{1}{\mu}
\log\left(\frac{c_1}{c_k}x_{k-1}\right), \quad k=2,\ldots,d,
$$
are well-defined for $T_0\gg T_1\gg 1$.
This produces the desired family of cylinders in
$\mathcal M^{J}_{0,2}(\alpha_1^d;\alpha_2^d)$.
If we now simultaneously replace each $v_{2k}$ by its opposite plane
$v_{2k}'$, then $c_k$ is replaced by $c'_k=-c_k$ for all $k$.
Since
$
\frac{c'_k}{c'_1}=\frac{c_k}{c_1},
$
the linear solution $\bar x$ is unchanged, and the same argument
shows that gluing is still possible.  
\end{rem}

\subsection{Coherent orientations}

In this section we use the Hutchings--Taubes system of coherent orientations \cite{hutchings2009}
to orient the moduli space
$\M^J_{0,2}(\alpha_1^d;\alpha_2^d)$ of immersed index--$2$ cylinders with positive end at $\alpha_1^d$ and negative end
at $\alpha_2^d$, whose compactification has boundary consisting of two-level
holomorphic buildings as described in Section \ref{sec:gluing}.
If $D$ is a Fredholm operator, we denote its determinant line by
$$
\mathcal{O}(D)
=
\mathcal{O}(\ker D)\otimes \mathcal{O}(\mathrm{coker}\, D).
$$
Here, $D=Du$ denotes the normal Cauchy-Riemann operator obtained by linearizing the
$\bar\partial_J$ operator along an immersed $J$-holomorphic curve $u$, where the normal
bundle is identified with the contact structure near the ends. In particular, the asymptotic
operators at the punctures determine the asymptotic data needed to define orientations of these
Fredholm operators.

The Hutchings-Taubes homotopy invariant coherent orientation system provides a consistent choice of orientations of
such determinant lines, compatible with disjoint unions and, most importantly, with gluing: if
$D = D_{-} \# D_{+}$ is the linearized operator associated with the gluing of two
$J$--holomorphic curves, then the induced orientation of $\mathcal{O}(D)$ is determined naturally
from the orientations of $\mathcal{O}(D_{+})$ and $\mathcal{O}(D_{-})$, with the aid of a suitable exact sequence. As a consequence, if the moduli space is cut out transversely (so that the cokernel vanishes),
then these choices induce orientations on the corresponding moduli spaces of $J$-holomorphic curves.

Let $v\in \M^{J}_{0,2}(\alpha_{1}^d;\alpha_{2}^d)$ be an index-$2$ cylinder with a positive end at $\alpha_{1}$ and a negative end at
$\alpha_{2}$. Such a curve is always immersed and regular, and thus ${\rm coker}(Dv)=0$. The coherent orientation of $\Ker(Dv)\equiv T_v\M^{J}_{0,2}(\alpha_{1}^d;\alpha_{2}^d)$ induces an
orientation of the one-dimensional connected component
$\mathcal C \subset \M^{J}_{0,2}(\alpha_{1};\alpha_{2})/\R
$
containing $[v]$. If $\mathcal B=(v_{+},v_{-})$ is a building at the boundary of $\mathcal C$ we denote
by $\eo(\mathcal B)\in\{-1,+1\}$ the induced endpoint orientation.

Consider two buildings $\mathcal B= (v_{+},v_{-})$ and $\mathcal B'=(v_{+},v'_{-})$ at the boundary of a regular
one-dimensional space $\M^{J}_{0,2}(\alpha_{1};\alpha_{2})/\R$ of the following form.
The top level is a $(2-d)$-index curve
$$
v_{+}\in \M^{J}_{0,d+2}\bigl(\alpha_{1}^d;\alpha_{2}^d,\underbrace{\alpha_{3},\ldots,\alpha_{3}}_{d\ \text{times}}\bigr),
$$
i.e. $v_{+}$ has one positive end at $\alpha_1^d$ and $d+1$ negative ends, of
which one is asymptotic to $\alpha_{2}^d$, and $d$ are asymptotic to $\alpha_{3}$. Moreover, $v_+$ $d$-covers a pair of pants with a positive end at $\alpha_1$ and two  negative ends at $\alpha_2$ and $\alpha_3$.
The bottom level is a disjoint union of curves
$$
v_{-}=v_{-,0}\sqcup v_{-,1}\sqcup\cdots\sqcup v_{-,d},
\qquad
v'_{-}=v_{-,0}\sqcup v'_{-,1}\sqcup\cdots\sqcup v'_{-,d},
$$
where $v_{-,0}$ is the trivial cylinder over $\alpha_{2}^d$ and
$
v_{-,i},\,v'_{-,i}\in \M^{J}_{0,1}(\alpha_{3};\emptyset),i=1,\dots,d,
$
are planes asymptotic to the hyperbolic orbit $\alpha_{3}$ with $\CZ(\alpha_{3})=2$.

\begin{theorem}\label{thm: coherent_orientations}
Assume that for every $i=1,\dots,d,$ the planes $v_{-,i}$ and $v'_{-,i}$ approach
$\alpha_{3}$ through opposite directions.
Then the induced endpoint orientations of $\mathcal B$ and $\mathcal B'$ have opposite signs, i.e., 
$\eo(\mathcal B)=-\eo(\mathcal B').
$
\end{theorem}

\begin{proof}
Let $w_R=v_-\#_d v_+\in \M^{J}_{0,2}(\alpha_{1}^d;\alpha_{2}^d)$ and $w_R'=v_-'\#_d v_+\in \M^{J}_{0,2}(\alpha_{1}^d;\alpha_{2}^d)$ be the glued cylinders ($R\gg 1$).
Since the curves in $\M^J_{0,2}(\alpha_1^d;\alpha_2^d)$ are regular, ${\rm coker}(Dw_R)={\rm coker}(Dw_R')=0$.
Hence each $\Ker(Dw_R)$ is $2$-dimensional and contains the $\R$--translation direction.

Proposition 9.3 in \cite{hutchings2009} proves the existence of an exact sequence
$$
0\to \Ker(Dw_R)\xrightarrow{\,f_R\,}\Ker(Dv_-)\oplus\Ker(Dv_+)\xrightarrow{\,g_R\,}  {\rm coker}(Dv_+)\to 0, \quad R \gg 1,
$$
and similarly for $w_R'$. Thus the orientation of $\Ker(Dw_R)$ is identified with the orientation of a
$2$–dimensional subspace of
$
\Ker(Dv_-)\oplus \Ker(Dv_+)
$
that is $C^0$–close to the span of the translation directions of the two levels.

The orientation of $\mathcal M^J_{0,1}(\alpha_3; \emptyset)$ can be a priori fixed as the one induced by $\partial_a(v),$ $v\in \mathcal M^J_{0,1}(\alpha_3; \emptyset)$. The projection of $\partial_a(v)$ onto the normal bundle $N_v$ is an element $\widetilde \partial_a(v)$ in $\ker (Dv)$, and near $\alpha_3$, $\widetilde \partial_a$ lies in the contact structure and points towards $\alpha_3$ precisely in the direction of the leading eigenspace along $\alpha_3$. This follows from the asymptotic formula near $\alpha_3$.

Hence, for each plane $v_{-,i}$ let $\widetilde{\partial_a}(v_{-,i})$ denote the $L^2$--projection of the
translation vector field $\partial_a(v_{-,i})$ to $\ker(Dv_{-,i})$.
Define
$
\partial_a^- :=
(\widetilde{\partial_a}(v_{-,1}),\ldots,\widetilde{\partial_a}(v_{-,d}))
\in \Ker(Dv_-).
$
In the same way, let $\widetilde{\partial_a}(v_+)\in \ker(Dv_+)$ be the projected translation on
the top level.

Consider the ordered pair in $\Ker(Dv_-)\oplus\Ker(Dv_+)$ given by
$u_1 := (\partial_a^-,0)$ and $u_2 := (0,\widetilde{\partial_a}(v_+)).$
For $R\gg 1$, the subspace $f_R(\Ker(Dw_R))$ is close to $\R\langle u_1,u_2\rangle$, hence the
orientation of $\Ker(Dw_R)$ is the same as the orientation induced by the ordered pair $(u_1,u_2)$.
The analogous statement holds for $w_R'$ with $u_1$ replaced by
\[
u_1' := ((\partial_a^-)',0),
\qquad
(\partial_a^-)' :=
(\widetilde{\partial_a}(v_{-,1}'),\ldots,\widetilde{\partial_a}(v_{-,d}'))
\in \Ker(Dv_-').
\]

Since the planes $v_{-,i},v_{-,i}'$ approach $\alpha_3$ through opposite 
directions, the only change in the computation of the orientations of the corresponding ends of $\mathcal M^J_{0,2}(\alpha_1^d ; \alpha_2^d)$ is that the translation generators differ by a negative scalar, i.e.
$
\widetilde{\partial_a}(v_{-,i}')=\mu\,\widetilde{\partial_a}(v_{-,i}), \mu<0.
$
Hence, $u_1'=\mu\,u_1$ with $\mu<0$, while $u_2$ is unchanged.
Therefore, $(u_1,u_2)$ and $(u_1',u_2)$ determine opposite orientations, and thus the orientations of $\mathcal B$ and $\mathcal B'$ differ by a minus sign, i.e. 
$\eo(\mathcal B')=-\eo(\mathcal B)$. \end{proof}

 Therefore, the boundary contributions arising from buildings of type (iii.6) cancel in pairs and do not contribute to the boundary of the moduli space of index-2 cylinders.

    \begin{figure}[h]
        \centering
        \includegraphics[width=0.5\linewidth]{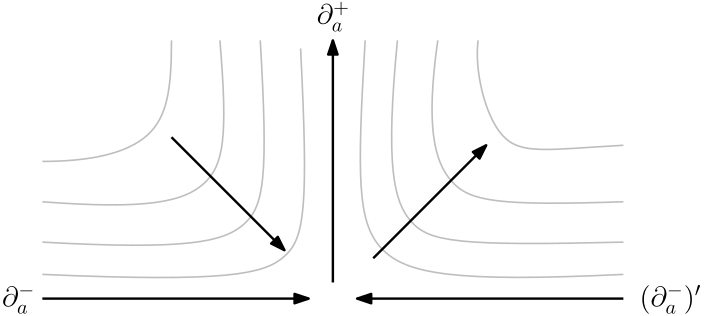}
        \caption{The buildings $\mathcal B$ and $\mathcal B'$ at the boundary of the compactification of  $\mathcal{M}^J_{0,2}(\alpha_1^d;\alpha_2^d)/\R$ have opposite signs. }
        \label{fig:orientation}
    \end{figure}

\section{The existence of the opposite plane}

The goal of this section is to prove the following proposition,
which provides geometric hypotheses on the weakly convex contact manifold so that every $J$-holomorphic plane $u:\C \to \R \times M$ asymptotic to an index-$2$ Reeb orbit $\alpha$ admits an `opposite' $J$-holomorphic plane $u':\C \to \R \times M$ which is asymptotic to $\alpha$ through opposite direction.

\begin{prop}\label{prop_opposite}
Let $(M,\xi=\ker \lambda)$ be a co-orientable, tight, compact, connected contact three-manifold with $\pi_2(M)=0$. Assume that $M$ has a strong filling $(W,\omega)$ satisfying the asphericity condition $\omega|_{\pi_2(W)}=0$ and so that the inclusion $i:\pi_1(M) \to \pi_1(W)$ is injective. Assume the $\lambda$ is a nondegenerate and weakly contact form and let $J \in \mathcal{J}_{\text{reg}}(\lambda)$. Let $\alpha$ be a simple unknotted Reeb orbit such that $\CZ(\alpha)=2$.
Then, either there exists no $J$-holomorphic plane asymptotic to $\alpha$, or, up to reparametrizations and translations in the $\R$-direction, there exist precisely two such planes, which approach $\alpha$ through opposite directions.
\end{prop}

\begin{proof} 
Let $(W,\omega)$ be a strong symplectic filling of $(M,\xi=\ker\lambda)$, i.e., $\partial W = M$ and there exists a Liouville vector field $Y$
defined near $\partial W$ pointing outward such that $\iota_Y \omega =  \lambda.$
Near the boundary we identify a collar neighborhood with
$((-\varepsilon,0]\times M, d(e^a\lambda)),$
where $a$ denotes the coordinate in the Liouville direction. We extend $W$ to its symplectic completion
$$
\widehat W
=
W \cup_{M} ([0,\infty)\times M),
$$
equipped with the symplectic form
$$
\widehat\omega =
\begin{cases}
\omega & \text{on } W, \\
d(e^a\lambda) & \text{on } [0,\infty)\times M.
\end{cases}
$$

Let $J$ be a cylindrical almost complex structure on
$\R\times M$ compatible with $d(e^a\lambda)$.
We choose an almost complex structure $\widehat J$ on $\widehat W$
such that:

\begin{itemize}
\item
$\widehat J$ coincides with $J$ on the cylindrical end
$[0,\infty)\times M$;

\item
$\widehat J$ is compatible with $\omega$ on the compact part $W$;

\item
$\widehat J$ is adjusted to the Liouville vector field near $M$.
\end{itemize}

Such an extension exists by standard arguments. We may assume that $\widehat J$ is generic in the following sense. If a curve is $\widehat J$-holomorphic and somewhere injective then it is Fredholm regular. This condition is achieved by modifying $\widehat J$ in a small collar neighborhood of $M\subset \widehat W$, see \cite[Theorem 7.2]{W16}.

 Let $v_0\in \M_{0,1}^{J}(\alpha;\emptyset)/\R$ be a $J$-holomorphic plane asymptotic to $\alpha$, and whose image is entirely contained in $[1,+\infty) \times M$. Up to translation, we may assume that  $v_0$ intersects $\{1\}\times M$. Let $X$ be the connected component of $\M_{0,1}^{\widehat{J}}(\alpha;\widehat{W})$ containing $v_0$. Since all curves in $X$ have index $1$ and are somewhere injective, the generic condition on $\widehat J$ implies that $X$ contains only regular curves and, in particular, contains the positive $\R$-translations of $v_0$. Moreover, $X$ is not diffeomorphic to $S^1$, and hence $X$ is diffeomorphic to an open interval. Notice that the energies of the curves in $X$ are uniformly bounded. 

Let $(u_n)_{n\in \N}\subset X$ be a sequence  converging to the end of $X$ that is not the end corresponding to the positive $\R$-translations of $\bar{u}$. By the SFT Compactness Theorem, $u_n$ has a subsequence (still denoted $u_n$) that converges to a building $\mathcal B=(v_1,\ldots,v_N)$ with arithmetic genus zero, so that $\mathcal B$ has only one positive puncture at $\alpha$. Here, the levels consist of $J$-holomorphic curves except perhaps the level at the bottom which may be $\widehat J$-holomorphic. The building $\mathcal B$ has the structure of a rooted tree where $v_1$ is the root with precisely one positive end at $\alpha$. Notice that $v_1$ may be either $\widehat J$ or $J$-holomorphic. 

The  symplectic asphericity condition implies that if the bottom level is   $\widehat J$-holomorphic, then it contains no non-trivial $\widehat J$-holomorphic sphere. Indeed, any non-trivial $\widehat J$-holomorphic sphere of the bottom level necessarily lies in $W$ by the maximum principle, and thus cannot exist.

Assume first that $v_1$ is $\widehat J$-holomorphic. By the symplectic asphericity condition,  $v_1$ contains no $\widehat J$-holomorphic sphere, and this implies that $v_1$ is a somewhere injective index-$1$ $\widehat J$-holomorphic plane.   Hence, $v_1$ is regular and $N=1$. Moreover, the generic assumption on $\widehat J$ implies that around $v_1$ there exists a one-parameter family of $\widehat J$-holomorphic planes asymptotic to $\alpha$. This contradicts the fact that $(u_n)_{n\in \N}$ lies at the end of the component $X\subset \mathcal M^{\widehat J}_{0,1}(\alpha ; \widehat W)$ as $n\to \infty$. We conclude that $v_1$ is necessarily $J$-holomorphic. We have to show that $v_1$ is the only level of $\mathcal B$ and it is formed by an index-$1$ $J$-holomorphic plane asymptotic to $\alpha$.

Denote by $v_1:\dot \Sigma \to \R \times M$ the $J$-holomorphic at the top of $\mathcal B$, asymptotic to $\alpha$ at its unique positive end, and assume that $v_1$ has $p\geq 0$ negative ends at the Reeb orbits $\beta_1,\ldots, \beta_p$. Using the lower levels of $\mathcal B$ to cap each $\beta_i$, we conclude that $\beta_i$ is contractible in $W$ for every $i$. Since $i:\pi_1(M)\to \pi_1(W)$ is injective, this implies that $\beta_i$ is also contractible in $M$. Since $\lambda$ is weakly convex and all $\beta_i$ are contractible im $M$, we have $\CZ(\beta_i)\geq 2 \forall i.$ Since $v_1$ is somewhere injective and is not a trivial cylinder, we know that $v_1$ is regular, which implies that $\ind (v_1) \geq 1$. Hence,
$$
\begin{aligned}
1\leq \ind(v_1) & = -\chi(\dot \Sigma) + 2c_1^\tau(v_1^*\xi)+\CZ(\alpha) - \sum_{i=1}^p \CZ(\beta_i) \\& = -(2-1-p)+2-\sum_{i=1}^p \CZ(\beta_i) = p+1-\sum_{i=1}^p \CZ(\beta_i)\\ & \leq p+1-2p =-p+1 \leq 1,
\end{aligned}
$$
where we have used that $\pi_2(M)=0$ in the computation above. This implies $p=0$ and thus $v_1$ is a $J$-holomorphic plane asymptotic to $\alpha$. In this case, $N$ is necessarily equal to $1$ since there cannot exist any level in $\mathcal B$ below $v_1$, and thus $\mathcal B = (v_1).$

Since $v_1$ is $J$-holomorphic, we have shown that the sequence $(u_n)_{n\in \N}$ escapes to infinity as $n\to \infty$. Since $J$ is regular and $\CZ(\alpha)=2$, $v_1$ lies in a one-parameter family of such planes given by the $\R$-translations of $v_1$. Thus $u_n$ coincides with one such translation of $v_1$ for $n$ sufficiently large. By construction $v_1$ cannot coincide with $v_0$ in $\mathcal M^J_{0,1}(\alpha ; \emptyset)/\R$ since the $\R$-translations of $v_0$ correspond to the other end of $X$. Hence, 
\begin{equation}\label{v0neqv1}
    [v_0] \neq [v_1] \in \mathcal M^J_{0,1}(\alpha ; \emptyset) / \R.
\end{equation}

We still need to show that $v_0$ and $v_1$ approach $\alpha$ through opposite directions. We first show that $v_0$ and $v_1$ do not intersect $\alpha$. Recall the following theorem by Hofer, Wysocki and Zehnder.

\begin{theorem}[Hofer, Wysocki, Zehnder {\cite[Theorem 1.3]{prop2}}] \label{prop2,thm1.3}
Let $(M,\lambda)$ be a compact, connected (orientable) contact three-manifold, where $\xi = \ker \lambda$ is a tight contact structure. Let $J \in \mathcal{J}(\lambda)$ and $u:=(\bar{a},\bar{u}):\C \to (\R \times M,J)$ be a non constant finite-energy $J$-holomorphic plane with positive end asymptotic to a Reeb orbit $\alpha$ such that the covering number of $u$ satisfies $d(u)=k$. Suppose that $\alpha$ is nondegenerate, $k$-unknotted with respect to the relative homology class of $u$, and $\CZ(\alpha)\leq 3$, where $\tau$ is a trivialization of $\xi$ along $\alpha$ induced by $\bar u$. Then $\bar{u}(\C) \cap \alpha(\R)=\emptyset$.
\end{theorem}

Since $\alpha$ is unknotted and $\pi_2(M)=0$, we conclude from Theorem \ref{prop2,thm1.3} that both $v_0$ and $v_1$ are in the same relative class of the unknots bounded by $\alpha$ and thus do not intersect their asymptotic limit $\alpha$.

We also need the following theorem to show that the projections of $v_0$ and $v_1$ to $M$ are disjoint. 

\begin{theorem}[Hofer, Wysocki, Zehnder {\cite[Theorem 1.4]{prop2}}] \label{prop2,thm1.4}
Let $(M,\lambda)$ be a closed contact three-manifold and $J \in \mathcal{J}(\lambda)$. Let $u:=(\bar{a},\bar{u}),v:=(\bar{b},\bar{v}):\C \to \R \times M$ be two finite-energy $J$-holomorphic planes asymptotic to the same non-degenerate Reeb orbit $\alpha$. Assume that:
\begin{itemize}
\item the Chern number of the complex bundle $(\bar u \sharp \bar v)^* \xi \to S^2$ computed on the fundamental class of $S^2$ is zero, where $\bar u \sharp \bar v: S^2=\C \cup \overline{\C} \to M$ and $\overline{\C}$ is the set of complex numbers with the opposite orientation (with respect to the domain of $\bar u$) along the circle $\alpha$ at infinity;
\item $d(u)=d(v)$;
\item $\CZ(\alpha) \leq 3$;
\item $\bar{u}(\C) \cap \alpha(\C)=\bar{v}(\C) \cap \alpha(\C)=\emptyset$.
\end{itemize}
Then $\bar{u}(\C)=\bar{v}(\C)$ or $\bar{u}(\C) \cap \bar{v}(\C) = \emptyset.$
\end{theorem}

Since $[v_0] \neq [v_1]$, we conclude from Theorem \ref{prop2,thm1.4} that $\bar v_0(\C) \cap \bar v_1(\C) = \emptyset,$ where $v_0 = (\bar a_0, \bar v_0)$ and $v_1 = (\bar a_1, \bar v_1)$. 

Finally, Siefring's intersection theory \cite{S11} implies that that if two $J$-holomorphic planes converge to $\alpha$ through the same direction, then their projections must intersect, see \cite{PS2018} for a discussion. We conclude that $v_0$ and $v_1$ approach $\alpha$ through opposite directions. By the discussion above, we also conclude that no third plane exists. This finishes the proof of Proposition \ref{prop_opposite}. \end{proof}

\section{Proof of main results} \label{sec:homologia}
Let $(M, \xi = \ker \lambda)$ be a closed contact three-manifold. Assume that $\lambda$ is nondegenerate and let $J\in \mathcal J(\lambda)$. The contact homology $\text{CH}(M,\lambda, J)$ is defined as follows. Let $\text{CC}(M,\lambda)$ be the vector space over $\Q$ generated by the good Reeb orbits $\mathcal P_{\rm good}(\lambda)$, i.e., by those orbits which are not even covers of negative hyperbolic orbits. Let $J \in \mathcal{J}(\lambda)$. 
The linear operator $\partial^J:\text{CC}(M,\lambda) \to \text{CC}(M,\lambda)$ is defined on the generators by  
\begin{equation}
\partial^J(\alpha)=\sum_{\beta \in \mathcal{P}_{\rm good}(\lambda)} \sum_{u \in \M_{0,1}^{J} (\alpha;\beta)/\R} \epsilon(u) \frac{d(\alpha)}{d(u)}  \beta, \quad \forall \alpha \in \mathcal{P}^g(\lambda).
\label{eq:delta}
\end{equation}
Here, $ \epsilon(u) \in \{-1,1\}$ is a sign associated with $u$, which is given by a system of coherent orientations, see \cite{hutchings2009} or \cite{BM}, $d(\alpha) \in \N$ is the covering number of $\alpha$, and $d(u)\in \N$ is the covering number of $u$, see \eqref{d2}. Recall that $u\in \M_{0,1}^J(\alpha;\beta)$ is an index-$1$ $J$-holomorphic cylinder with one positive end at $\alpha$ and one negative end at $\beta$.

The map $\partial^J$ is well-defined when the sum on the right side of \eqref{eq:delta} is finite, i.e., for each Reeb orbit $\alpha \in \mathcal{P}^g(\lambda)$, the set
\begin{equation}\label{BB}
\mathcal{B}:=\left\{\beta \in \mathcal{P}_{\text{good}}(\lambda) \; ; \;  \exists u \in \M_{0,1}^{J}(\alpha; \beta)/\R \right\}
\end{equation}
is finite. This is indeed true by the following Lemma.

\begin{lemma}[Hutchings-Nelson {\cite[Lemma 4.2]{HN}}] \label{lemma4.2}
Let $(M,\xi = \ker \lambda)$ be a closed three-contact manifold, where $\lambda$ is nondegenerate, and let $J \in \mathcal{J}_{\text{reg}}(\lambda)$. Then
\begin{itemize}
\item[(i)] $\M_{0,1}^{J}(\alpha;\beta)/\R$ is a 0-manifold for any Reeb orbits $\alpha,\beta \in \mathcal {P}(\lambda)$.

\item[(ii)] $\M_{0,2}^{J}(\alpha;\beta)/\R$ is a 1-manifold whenever $\alpha,\beta \in \mathcal{P} _{\text{good}}(\lambda)$.

\item[(iii)] If $\alpha,\beta \in \mathcal{P}_{\text{good}}(\lambda)$, then the mapping
$
d:\M_{0,2}^{J}(\alpha;\beta)/\R \to  \N^*:
u \mapsto  d(u),
$
is locally constant.
\end{itemize}
\end{lemma}
If $\partial^J \circ \partial^J=0$, then the cylindrical contact homology $\text{CH}(M,\lambda,J)$ is defined as the homology of the chain complex $(CC(M,\lambda), \partial^J)$. 

Assume that the conditions of Theorem \ref{main1} hold. We want to show that $\partial^J \circ \partial^J=0$. This is proved by Hutchings and Nelson \cite{HN} in the case $\lambda$ is dynamically convex, see Theorem \ref{thm:HN}. In the current situation, we admit orbits with Conley-Zehnder index $2$. 

We consider good Reeb orbits $\alpha,\beta$ as in Lemma \ref{lemma4.2}-(ii). Consider the building $\mathcal{B}$ appearing at the boundary of the SFT-compactification of a component $X\subset \mathcal M^J_{0,2}(\alpha ; \beta)/ \R$, see Theorem \ref{C}. Hutchings and Nelson resolve all cases for which $\mathcal B$ contains only Reeb orbits with Conley-Zehnder index $\geq 3$. We have to deal with the buildings $\mathcal B$ containing Reeb orbits with Conley-Zehnder index $2$. The treatment given to the buildings in Theorem \ref{C}-(iii.1)-(iii.3) is the same as in \cite{HN} if an index-$2$ orbit appears in one of these buildings. So we only need to consider the buildings in (iii.4), (iii.5) and (iii.6). The buildings in (iii.4) and (iii.5) are ruled out in section \ref{sec_ruling_out_buildings}, see Propositions \ref{naopode1} and \ref{naopode2}.  

Hence, we assume that $\mathcal B$ is as in Theorem \ref{C}-(iii.6). Since $J$ is generic, $\mathcal B=(v_1,v_2)$ has precisely two levels, and there exists $d\in \N$ so that the first level $v_1$ consists of a $(2-d)$-index curve asymptotic to $\alpha=\alpha_1^d$, a negative end at $\beta=\alpha_2^d$ and $d$ negative ends at an index-$2$ hyperbolic orbit $\alpha_3$. This curve $d$-covers a pair of pants with a positive end at $\alpha_1$ and two negative ends at $\alpha_2$ and $\alpha_3$. The second level $v_2$ consists of a trivial cylinder over $\alpha_2^d$ and $d$ planes asymptotic to $\alpha_3$. The system of coherent orientations assigns a sign ${\rm eo}(\mathcal B) \in \{-1,+1\}$ to $\mathcal B$ as a boundary point of $X$. Consider the building $\mathcal B'=(v_1,v_2')$ as in (iii.6), whose first level coincides with $v_1$, but on the second level $v_2'$, we replace the planes of the second level of $\mathcal B$ with their respective opposite pairings. By Theorem \ref{thm:gluing-opposite}, $\mathcal B'$ can be glued to obtain a new component $X' \subset \mathcal{M}^J_{0,2}(\alpha ; \beta)$ for which $\mathcal B'$ is a boundary point of $X'$. By Theorem \ref{thm: coherent_orientations}, we have ${\rm eo}(\mathcal B') = -{\rm eo}(\mathcal B)$. Hence, the components $X$ and $X'$ can be glued together at $\mathcal B \equiv \mathcal B'$, keeping the correct orientation, so that these two boundary points are now invisible. Doing the same for all such pairs of building appearing in the compactification of $\mathcal{M}^J_{0,2}(\alpha ; \beta),$ we see that the proof that $\partial^J \circ \partial^J =0$ reduces to the dynamically convex case, considered by Hutchings and Nelson in \cite{HN}. This finishes the proof of Theorem \ref{main1}.    

Theorem \ref{main2} now follows directly from Theorem \ref{main1} and Proposition \ref{prop_opposite}.

\section*{Acknowledgement}

 PS was partially supported by the National Natural Science Foundation of China (grant number W2431007). AKO was partially supported by the Brazilian Coordination for the Improvement of Higher Education Personnel - CAPES (grant number 88882.377930/2019-01). AKO and PS thanks the support of the Shenzhen International Center for Mathematics - SUSTech.

\bibliographystyle{plain}
\bibliography{bibliografia.bib}

\end{document}